\documentclass[11pt,a4paper]{article}
\usepackage[a4paper,margin=2.5cm]{geometry}
\usepackage{amsmath,amssymb,amsfonts,amsthm}

\author{Marc Troyanov}
\date{January 7, 2026}
\title{The Choreography of Geodesics in SOL}
\usepackage{appendix} 
\numberwithin{equation}{section}
\usepackage[hidelinks]{hyperref}
\usepackage{color}
\usepackage{enumerate}
\usepackage{caption}

\usepackage{makeidx}
\makeindex

\usepackage{fancyhdr}
\fancypagestyle{plain}{%
  \fancyhf{}%
  \fancyfoot[L]{\hrule \vspace{1pt}  \footnotesize Marc Troyanov,   {\'E}cole Polytechnique Federale de Lausanne \ (\texttt{marc.troyanov@epfl.ch}) }%
}

\usepackage[T1]{fontenc}
\usepackage[utf8]{inputenc}
\usepackage[english]{babel}

\usepackage{amsmath,amssymb,amsfonts,amsthm}
\usepackage{mathrsfs} 
\usepackage{mathtools} 

\usepackage{graphicx} 
\usepackage{pgf,pgfplots,pgfmath}
\usepackage{pgf,tikz,tikz-3dplot}
\usetikzlibrary{decorations.markings}  
\pgfplotsset{compat=newest}  
 
\newcounter{nextnum}[section] 

\theoremstyle{plain}
\newtheorem{theorem}{\rm\bf Theorem}[section]
\newtheorem{proposition}[theorem]{\rm\bf Proposition}
\newtheorem{lemma}[theorem]{\rm\bf Lemma}
\newtheorem{corollary}[theorem]{\rm\bf Corollary}
\newtheorem{remark}[theorem]{\rm\bf Remark}
\newtheorem{definition}[theorem]{\rm\bf Definition}

\newcommand{\R}{\mathbb{R}}

\newcommand{\ee}{\mathrm{e}}

\DeclareMathOperator{\arccosh}{arccosh}

\DeclareMathOperator{\arctanh}{arctanh}

\DeclareMathOperator{\sgn}{sgn}

\DeclareMathOperator{\SOL}{SOL}
\DeclareMathOperator{\NIL}{NIL}
\DeclareMathOperator{\cut}{cut}
\DeclareMathOperator{\Cut}{Cut}

\begin{document}
 
\maketitle

\begin{abstract}
We provide a self-contained geometric description of the geodesic flow in the three-dimensional Lie group $\SOL$, one of Thurston's eight model geometries. The geometry of geodesics is governed by a single invariant $k\in [0,1]$, its \textit{modulus}. Generic geodesics spiral around an axis, with  well-defined amplitude $A(k)$, period $T(k)$, and horizontal drift $H(k)$. We characterize minimal geodesic segments and the cut locus, and obtain an asymptotic estimate showing that the distances between points at the same altitude grow logarithmically. This work builds on previous work by Grayson and Coiculescu-Schwartz, but develops an alternative geometric and dynamical viewpoint.

\medskip 


\smallskip 

\end{abstract}

\medskip 

\setcounter{tocdepth}{2}
{\small  \tableofcontents  }

\newpage 

\section{Introduction} 
In the late 1970s, William P. Thurston formulated his Geometrization Conjecture (now a theorem, thanks to Perelman’s proof), asserting that any closed, connected $3$-manifold decomposes along embedded $2$-spheres and incompressible $2$-tori into elementary pieces, each carrying a complete Riemannian metric locally isometric to one of  eight model geometries.\footnote{See Scott~\cite{Scott} for a classic account of the eight geometries, and Bonahon~\cite{Bonahon} for a  survey of geometric structures on $3$-manifolds centered on Thurston’s program.}
Among these, the geometry of $\SOL$ governs torus-bundles over the cirle  with Anosov monodromy.
It was in the highly creative atmosphere that prevailed at Princeton in the early 1980's that Matthew A. Grayson wrote his thesis ~\cite{Grayson} under Thurston's supervision. Chapter 3 of the thesis is devoted to the geodesic flow of $\SOL$ and closed $\SOL$-manifolds, and describes the possible lengths of closed geodesics representing free homotopy classes in toric fiber bundles.

\smallskip 

Reflecting on Thurston's generous mentorship and the delightfully informal yet creative atmosphere 
he inspired, Grayson writes in the acknowledgments of his thesis:
\begin{center}
\begin{minipage}{0.93\textwidth}
\begin{quote}\small\ttfamily
``Bill turns out another graduate student. I don't see why he does it.
We bother him, ask him dumb questions, follow him to Boulder, ask him for letters of
recommendation, and mostly fail to understand what he's trying to say.''
\end{quote}
\end{minipage}
\end{center}

Since Grayson’s thesis, several works have sharpened our understanding of $\SOL$.
In particular, M. P. Coiculescu and R. E. Schwartz analyze in ~\cite{CoiculescuSchwartz}   the cut locus and the topology of metric spheres in $\SOL$, and an earlier paper by the present author~\cite{Troyanov98} describes the asymptotic behavior of geodesics and the horizon of $\SOL$.

\medskip

  $\SOL$ is defined as the semidirect product $\R^{2}\rtimes\R$ determined by the anisotropic scaling action of $\R$ on $\R^{2}$, given explicitly by  
\[
  (x,y)\mapsto (\ee^{z}x,\; \ee^{-z}y).
\]
It is a solvable Lie group, and as a manifold it is diffeomorphic to $\R^{3}$. We endow $\SOL$ with the following standard left-invariant Riemannian metric
\[
  ds^2 = \ee^{-2z}\,dx^{2} \;+\; \ee^{2z}\,dy^{2} \;+\; dz^{2}.
\]
 
The geometry of $\SOL$ displays a rich internal structure. Its isometry group reflects the underlying anisotropy, and the manifold admits several natural foliations. In particular, there are vertical foliations by totally geodesic hyperbolic planes, and a horizontal foliation by Euclidean planes, which however are neither totally geodesic nor isometrically embedded. 

  Since $\SOL$ is a Lie group with a left–invariant metric, its geodesic flow is a Hamiltonian system with abundant symmetries. By Noether’s theorem, these symmetries yield conserved quantities, making the system \emph{completely integrable}. The geodesic flow can then be solved by the classical \emph{method of quadratures}, that is, by a finite sequence of explicit integrations combined with elementary algebraic manipulations\footnote{For background on completely integrable Hamiltonian systems see \cite{Arnold1,BolsinovFomenko,CushmanBates,Torrielli}; for Noether’s theorem see \cite{KosmannNoether}.}.  
 
  For a given geodesic $\gamma(t)=(x(t),y(t),z(t))$ in $\SOL$, the quantities
\[
  a =  \ee^{2z}\dot{x}, \qquad b = \ee^{-2z}\dot{y}
\]
remain constant; we call them the \emph{principal constants of motion}. The vertical coordinate then satisfies the autonomous ODE
\[
  \ddot{z} + U'_{a,b}(z) = 0, \qquad 
  U_{a,b}(z) = \tfrac12\bigl(a^{2}\ee^{2z} + b^{2}\ee^{-2z}\bigr).
\]
Solving the geodesic equations thus reduces to integrating this one–dimensional differential equation for $z(t)$ and recovering the horizontal coordinates by quadrature, that is, by direct integration:\footnote{About the history of the method of quadratures applied to geodesics, see Nabonnand~\cite{Nabonnand1995}.}
\[
  x(t) = \int a\,\ee^{-2z(t)}\,dt, \qquad
  y(t) = \int b\,\ee^{2z(t)}\,dt.
\]
Exact solutions can be expressed in terms of Jacobi elliptic functions as in  \cite{Troyanov98}, however in the present paper, we shall not need these explicit formulas.
 A direct computation shows  that the quantity
\[
  c = a\,x(t) - b\,y(t) + \dot{z}(t)
\]
is also conserved, i.e. $\dot c = 0$,  providing a third important constants of motion for the geodesic flow. 
 
A fundamental geometric invariant for the classification of geodesics is the \emph{modulus} $k\in [0,1]$, defined for a unit–speed geodesic $\gamma$ by
\[
  k = \sqrt{\frac{1 - 2\lvert ab\rvert}{1 + 2\lvert ab\rvert}},
\]
where $a$ and $b$ are the principal constants of motion of $\gamma$. In Proposition~\ref{kisinvariant2} we establish that  
\emph{two generic geodesics in $\SOL$ are geometrically equivalent if and only if they have the same modulus $k$.}
We call a geodesic \emph{special} when $k\in\{0,1\}$, and \emph{generic} when $0<k<1$.  Special geodesics are easily described: they are either horizontal lines lying in some horizontal plane and bisecting the coordinate axes, or they lie entirely in vertical, totally geodesic hyperbolic planes.  
We now turn to the case of generic geodesics.

The vertical coordinate $z(t)$ of a generic geodesic oscillates periodically with {amplitude} $A$ determined by the condition $2U_{a,b}(A)=1$; in terms of the modulus, we have 
$A(k) = \arctanh(k)$. 
As shown in Section~\ref{sec.TMHelliptic}, the oscillation period is $T(k) = \sqrt{8\,(1 + k^{2})}\,K(k)$, where $K(k)$ denotes the complete elliptic integral of the first kind

Following Grayson, to each generic geodesic $\gamma$ with constants of motion $a,b,c$ we associate its \emph{Grayson cylinder}, defined by
\[
  \mathcal G_{a,b,c} = \bigl\{ (x,y,z) \in \SOL \;\bigm|\; (a x - b y - c)^{2} + 2 U_{a,b}(z) = 1 \bigr\}.
\]
This surface is diffeomorphic to a cylinder $S^{1}\times\R$, and $\gamma$ is entirely contained in it. Each Grayson cylinder is foliated by geodesics with the same modulus $k$, and it \emph{guides} the motion: the trajectory oscillates in height and winds once around the cylinder after each  period $T=T(k)$.
The cylinder admits a distinguished horizontal \emph{equatorial plane}, given by $\{z=h\}$,  where 
\[
  h = \frac{1}{2}\log\!\left|\frac{b}{a}\right|,
\]
represents the average height of the oscillating vertical coordinate $z(t)$. The surface is symmetric with respect to this plane: the horizontal \emph{inflection lines} (where $\ddot z=0$) lie in it, while the \emph{critical lines} (where $\dot z=0$) occur at the altitudes $z=h\pm A(k)$. Finally, the modulus $k$ is an invariant of the cylinder itself, independent of the particular geodesic it contains.

To frame the geometry of the Grayson cylinders, we introduce the following  global vector fields:
\[
\xi  = a \ee^{2z}\,\frac{\partial}{\partial x} + b \ee^{-2z}\,\frac{\partial}{\partial y} + (c-ax+by)\,\frac{\partial}{\partial z},
\qquad 
\eta = b\,\frac{\partial}{\partial x} + a\,\frac{\partial}{\partial y}.
\]
Both are tangent to the cylinder $\mathcal G_{a,b,c}$. The integral curves of $\xi$ are precisely the geodesics on $\mathcal{G}$, while those of $\eta$ are horizontal lines parallel to the axis. The  flows of $\xi$ and $\eta$ commute, they preserve the cylinder, and their angle is uniformly bounded away from $0$.

A generic geodesic winds around its associated Grayson cylinder with a constant horizontal drift after each period. For a generic unit–speed geodesic $\gamma$, we introduce the quantity $H$ defined by 
\[
   H^2 = {\,|x(t_0+T(k)) - x(t_0)| \; |y(t_0+T(k)) - y(t_0)|\,},
\]
which measures the net horizontal displacement after one oscillation. This expression is independent of the initial time $t_0$, and depends only on the modulus $k$.  More precisely we have 
\[
   H = H(k) = \frac{4E(k) - 2(1-k^{2})K(k)}{\sqrt{1-k^{2}}},
\]
where $E(k)$ denotes the complete elliptic integral of the second kind. We call $H(k)$ the \emph{horizontal drift invariant} of the geodesic $\gamma$, or the \emph{holonomy invariant} in the terminology of  \cite{CoiculescuSchwartz}.

\smallskip 
 
A fundamental property of the generic geodesics may be phrased as a \emph{rendez–vous theorem}. It asserts that for any generic unit–speed geodesic $\gamma$ and any time $t_{1}$, the points $\gamma(t_{1})$ and $\gamma(t_{1}+T(k))$ are either conjugate along $\gamma$, or else they can be joined by two distinct geodesic arcs of length $T(k)$. 

 The geometric mechanism is as follows: sliding the Grayson cylinder $\mathcal G_{a,b,c}$ horizontally to another cylinder $\mathcal G_{a,b,c'}$ (with the same constants $a,b$) produces two cylinders intersecting along distinct horizontal lines. Along these lines the horizontal components of the velocity coincide, while the vertical component agrees up to sign. This yields a natural \emph{partner} geodesic $\gamma^{\star}$, which intersects $\gamma$ again after one period $T(k)$.

\smallskip 
 
As a consequence of the rendez--vous phenomenon, a generic geodesic segment ceases to be minimizing once its length exceeds $T(k)$. The converse is a much deeper fact: it is the main theorem of Coiculescu and Schwartz~\cite{CoiculescuSchwartz}, who proved that a generic geodesic segment is minimizing precisely up to length $T(k)$, where $k$ is its modulus.

The same paper also provides a complete description of the cut locus,  beyond what we can 
summarize here. Let us just mention that the cut length is infinite for vertical or hyperbolic 
geodesics, equals $T(k)$ for generic ones, and in the horizontal case reduces to $\sqrt{2}\pi$, 
the infimum of $T(k)$. In all cases, the cut locus of a point is contained in its horizontal plane; 
in particular, $\Cut(0)\subset  \{z = 0\}$.

In addition to  these results, we  establish an asymptotic formula for large distances 
along fixed horizontal directions. Specifically, if $P_\lambda=\lambda(\cos\theta,\sin\theta,0)$, 
then the distance in $\SOL$ from the origin to $P_\lambda$ satisfies
\[
 d(0,P_\lambda)=4\log \lambda+O(1),
\]
as $\lambda\to\infty$.

\smallskip 

Finally, although our focus is on $\SOL$,  we include in Section~\ref{sec.nil} a brief  discussion of geodesics in the three-dimensional Heisenberg geometry $\NIL$, viewed from the same quadrature viewpoint. This both illustrates the general Lie-group/left-invariant framework and highlights which features are specific to $\SOL$.

\smallskip 
   
Our contribution is a self–contained geometric account of the geodesic flow in $\SOL$, based on quadratures and emphasizing invariant surfaces (the Grayson cylinders), the modulus as a complete invariant, and the winding/drift mechanism. While our results are consistent with those of Grayson \cite{Grayson} and Coiculescu–Schwartz \cite{CoiculescuSchwartz}—and indeed build on a substantial part of their work—our approach is different, and we hope it offers a complementary perspective on the geodesic flow in $\SOL$.

\smallskip 

 \textit{Organization of the paper.}  
The paper is organized as follows.  
Section~\ref{sec:background} recalls basic facts about $\SOL$ and sets up the notation.  
Section~\ref{sec:geodesics} introduces the geodesic equation and provides an initial classification into special and generic cases.  
Section~\ref{sec:generic} develops the core of the paper: the vertical periodicity, the geometry of the Grayson cylinders, the role of the modulus as a complete invariant, and the description of winding and drift, culminating in the rendezvous theorem.  
Section~\ref{sec:cutlocus} applies these results to the cut locus, including the logarithmic asymptotics for large distances.  
Two appendices complement the text: Appendix~A collects basic facts about elliptic integrals and their relation to the invariants $T(k)$ and $H(k)$, while Appendix~B provides a concise account of geodesics in the Heisenberg geometry $\NIL$.

\smallskip 

\textit{Acknowledgement.} The author thanks M.~P. Coiculescu, R.~E. Schwartz, and V.~Matveev for their comments on this chapter.

\section{Background on $\SOL$}\label{sec:background} 

By definition, the Lie group $\SOL$\index{SOL} is the space $\mathbb{R}^3$ endowed with the multiplication
$$
(x_1, y_1, z_1) \ast (x_2, y_2, z_2)  =  \bigl(x_1 + \ee^{z_1} x_2 ,\; y_1 + \ee^{-z_1} y_2,\; z_1 + z_2 \bigr).
$$
It is a semidirect product $\mathbb{R}^2 \rtimes \mathbb{R}$. The neutral element is the origin $O = (0,0,0)$, and a faithful linear representation is given by the map
$$
  \SOL \to \mathrm{GL}_3(\mathbb{R}), \qquad  (x,y,z) \mapsto 
 \begin{pmatrix}
\mathrm{e}^z & 0 & x \\
0 & \mathrm{e}^{-z} & y \\
0 & 0 & 1
 \end{pmatrix}.
$$
The left-invariant vector fields on $\SOL$ form its Lie algebra; a standard basis is 
\begin{equation}\label{FrameXYZ}
  X = \ee^z \frac{\partial}{\partial x}, \quad
  Y = \ee^{-z} \frac{\partial}{\partial y}, \quad
  Z = \frac{\partial}{\partial z},
\end{equation}
with Lie brackets
$$
  [Z, X] = X, \quad
  [Z, Y] = -Y, \quad
  [X, Y] = 0.
$$
This Lie algebra is solvable but not nilpotent, hence $\SOL$ is a solvable, non nilpotent,  Lie group.

\medskip

The linear map ($(1,1)-$tensor) defined on the Lie algebra  by 
$$
  X^{\dagger} = X, \quad
  Y^{\dagger} = Y, \quad
  Z^{\dagger} = -Z
$$
is called the \textit{vertical flip}. It is an  \emph{antimorphism} of the Lie algebra,  that is, 
$$[A, B]^{\dagger} = [B^{\dagger}, A^{\dagger}].$$

\bigskip 

The Riemannian metric on $\SOL$ defined by 
\begin{equation}\label{ds2}
 ds^2 = \ee^{-2z} dx^2 + \ee^{2z} dy^2 + dz^2
\end{equation}
is  left-invariant. The frame $\{X, Y, Z\}$ is orthonormal. 
With this metric, $\SOL$ is one of the eight model geometries of Thurston's classification, see the foundational survey by Peter Scott \cite{Scott}.
The left translations are isometries of $\SOL$. In particular we have
\begin{itemize}
\item[$\circ$]  {The horizontal translations}:
$$
    (x,y,z) \mapsto (x + w_1, y + w_2, z),
$$
\item[$\circ$]  {and the vertical lifts}:
$$
    (x,y,z) \mapsto \bigl(\ee^w x, \ee^{-w} y, z + w \bigr).
$$
\end{itemize}
Additionally, the isometry group contains some  symmetries fixing the origin:
\begin{itemize}
\item[$\circ$] the four coordinate sign changes
  $$
  (x, y, z) \mapsto (\varepsilon_1\,  x, \varepsilon_2 \, y, z), \quad \varepsilon_1, \varepsilon_2 \in \{\pm 1\},
  $$
\item[$\circ$] and the involutive isometry
  $$
  (x, y, z) \mapsto (y, x, -z).
  $$
\end{itemize}
These transformations generate the isotropy group at the identity, which  is a finite group of order 8, isomorphic to the dihedral group $D_4$. Together with translations and vertical lifts, they generate the full isometry group of $\SOL$.

\medskip

\begin{remark} \rm 
It is important to distinguish the points of $\SOL$ from the  tangent vectors, even though one might be tempted to identify both with $\R^{3}$. In particular, the vertical flip on tangent vectors satisfies $\|\xi^{\dagger}\|=\|\xi\|$, yet the map on points $(x,y,z)\mapsto(x,y,-z)$ is \emph{not} an isometry of $\SOL$. The operator ${\dagger}$ is a $(1,1)$–tensor and   it defines an isometric endomorphism of the tangent bundle, but this endomorphism is not the differential of any isometry of $\SOL$.
\end{remark}

\medskip 

$\SOL$ is a homogeneous but non-isotropic space, however   certain distinguished submanifolds provide useful insight into its geometry. In particular the vertical plane $\mathcal{H}'_0 = \{ y = 0 \}$ is pointwise fixed under the isometry $(x, y, z) \mapsto (x, -y, z)$ and is therefore a totally geodesic submanifold. It is isometric to the hyperbolic plane.\footnote{The coordinates $(x,z)$, in which the metric takes the form $ds^2 = \mathrm{e}^{-2z} dx^2 + dz^2$,  are commonly referred to as \emph{horocyclic coordinates} in  the hyperbolic plane.} Indeed, introducing coordinates  $u = x$, $v = \ee^z$, the induced metric on $\mathcal{H}'_0$ is
$$
\ee^{-2z} dx^2 + dz^2 = \frac{du^2 + dv^2}{v^2};
$$
this identifies $\mathcal{H}'_0$ with the Poincaré upper half-plane $\mathbb{H}^2$. Similarly, the subsets
$$
\mathcal{H}'_{y_0} = \{ y = y_0 \}, \qquad \mathcal{H}''_{x_0} = \{ x = x_0 \}
$$
are totally geodesic submanifolds of $\SOL$  isometric to $\mathbb{H}^2$ for any $x_0,y_0\in \R$.  The horizontal planes $\Pi_{z_0} = \{ z = z_0 \}$ also play an important role, albeit they are not  totally geodesic.

\medskip

In the left-invariant frame $\{X,Y,Z\}$, the Levi-Civita connection is computed in \cite{Troyanov98}.  
The only non-vanishing covariant derivatives are
\[
\begin{aligned}
\nabla_X X &= Z, \qquad &\nabla_X Z &= -X, \\[3mm]
\nabla_Y Y &= -Z, \qquad &\nabla_Y Z &= Y.
\end{aligned}
\]
The vertical flip $\dagger$ defined earlier satisfies the relation
\begin{equation*}\label{eq:dagger-connection}
  \nabla_{V^\dagger} V^\dagger = -\bigl(\nabla_V V\bigr)^\dagger.
\end{equation*}

\medskip

Recall that the Riemann curvature tensor $R$ is defined by
\[
R(U,V)W=\nabla_U\nabla_V W-\nabla_V\nabla_U W-\nabla_{[U,V]} W.
\]
Using the above formulas (and the fact that $[X,Y]=0$), we compute the sectional curvatures.
First,
\[
 R(X,Y)Y
   = \nabla_X \nabla_Y Y - \nabla_Y \nabla_X Y
   = -\nabla_X Z
   = -X,
\]
hence
\[
K(X,Y)=\langle R(X,Y)Y,X\rangle = 1.
\]
Similarly,
\[
K(X,Z)=K(Y,Z)=-1.
\]
These sectional curvatures agree with the fact that the vertical planes
$\mathcal{H}'_0$ and $\mathcal{H}''_0$ are totally geodesic and isometric
to the hyperbolic plane.

\medskip 

\begin{remark}\label{remarkHorizontal} \rm
We record here some convenient terminology and observations.

\smallskip

(1) For a point $p = (x,y,z)$, we  refer to the third coordinate $z$ as its
\emph{height} (or \emph{altitude}).

\smallskip

(2) Every isometry of $\SOL$ preserves the \emph{height difference}
between two points:
if $p_1 = (x_1,y_1,z_1)$ and $p_2 = (x_2,y_2,z_2)$, then
\[
|z_1 - z_2|
\]
is invariant under the full isometry group.

\smallskip

(3) Given two points $p_1 = (x_1,y_1,z)$ and $p_2 = (x_2,y_2,z)$ lying at the same altitude $z$, we define their
\emph{horizontal distance} by
\begin{equation}\label{horzdistance}
  \delta(p_{1},p_{2}) =\sqrt{\ee^{-2z}(x_{2}-x_{1})^{2}+\ee^{2z}(y_{2}-y_{1})^{2}}.
\end{equation}
This coincides with the Euclidean distance in the horizontal plane $\Pi_z = \{\, z = \mathrm{const} \,\}$ equipped with the metric induced by
\eqref{ds2}. In general, $\delta(p_1,p_2)$ is greater than the  Riemannian distance between $p_1$ and $p_2$, since $\Pi_z$ is not a totally geodesic submanifold.

\smallskip

(4) Every isometry of $\SOL$ maps each horizontal plane $\Pi_z$ affinely onto some horizontal plane, preserving Euclidean area (note that in a horizontal plane the Riemannian area coincides with the Euclidean one).

\smallskip

(5) In particular, if $p_1 = (x_1,y_1,z_1)$ and $p_2 = (x_2,y_2,z_2)$, then the product
\[
 |x_2 - x_1| \cdot |y_2 - y_1|
\]
is also preserved by every isometry of $\SOL$. 
\end{remark}

\section{Unveiling SOL's geodesics}\label{sec:geodesics}

\subsection{A first stroll  inside $\SOL$}

Let's imagine the Lie group $\SOL$ as a kind of large multi-story building, and suppose that we are in the lobby on the ground floor, situated at the origin $O = (0,0,0)$ and we crave for a drink at the bar, located at point $P = (p,p,0)$ \  ($p>0$).
One could of course  simply cut diagonally through the concourse along the path 
$$
  \alpha(t) = (t,t,0),\quad 0\le t\le p,
$$
which is indeed a geodesic of $\SOL$, with length $\sqrt2\,p$. But as $p$ grows, that path becomes excessively long. 
To reduce the distance, one should use the particular structure of the metric  \eqref{ds2}, which suggests  the following guideline:

\medskip

\begin{center}
  \begin{minipage}{0.85\linewidth}
    \textit{When on an upper floor, proceed west–east; when down  in the basement, move on south–north.}
  \end{minipage}
\end{center}

\medskip

More precisely, follow the following itinerary:

{\small
\begin{enumerate}[(i)]
\item From the main hall at $O=(0,0,0)$, take the elevator up to the upper floor at height $A$, keeping the $x$– and $y$–coordinates fixed at zero.
\item Walk along the west–east gallery to reach $(p,0,A)$, keeping $y=0$ and $z=A$ constant.
\item Take the elevator down to the basement level at $(p,0,-A)$, with $x=p$ and $y=0$ fixed.
\item Walk along the south–north corridor to $(p,p,-A)$, keeping $x=p$ and $z=-A$ constant.
\item Finally, take the elevator back to the ground floor to arrive at $P=(p,p,0)$.
\end{enumerate}
}

\tdplotsetmaincoords{60}{120}
\newcommand{\pval}{6}            
\pgfmathsetmacro{\hval}{ln(\pval/2)}  
\begin{center}

\begin{tikzpicture}[scale=0.7, tdplot_main_coords,
    axis/.style={->, thick},
    path3d/.style={ultra thick, blue},
    diag/.style={gray!50, very thick},      
    point/.style={fill=blue, circle, inner sep=1.5pt},
    label3d/.style={font=\small}
  ]

  \pgfmathsetmacro{\fac}{1.2}
  \pgfmathsetmacro{\xmin}{-0.2*\pval}
  \pgfmathsetmacro{\ymin}{-0.2*\pval}
  \pgfmathsetmacro{\xmax}{\fac*(\pval+1)}
  \pgfmathsetmacro{\ymax}{\fac*(\pval+1)}

  \filldraw[fill=gray!10, opacity=0.8, draw=none]
    (\xmin,\ymin,0) -- (\xmax,\ymin,0) -- (\xmax,\ymax,0) -- (\xmin,\ymax,0) -- cycle;

  \draw[axis] (0,0,0) -- (\pval+1,0,0) node[anchor=north east] {$x$};
  \draw[axis] (0,0,0) -- (0,\pval+1,0) node[anchor=north west] {$y$};
  \draw[axis] (0,0,0) -- (0,0,\hval+1) node[anchor=south]      {$z$};

  \coordinate (O) at (0,0,0);
  \coordinate (A) at (0,0,\hval);
  \coordinate (B) at (\pval,0,\hval);
  \coordinate (C) at (\pval,0,-\hval);
  \coordinate (D) at (\pval,\pval,-\hval);
  \coordinate (P) at (\pval,\pval,0);

  \draw[path3d]
    (O) -- (A) -- (B) -- (C) -- (D) -- (P);

  \draw[diag] (O) -- (P);
  
  \foreach \pt/\name in {O/$O$, P/$P$}{
    \fill[black] (\pt) circle (1.8pt);
    \node[label3d] at (\pt) [anchor=south west] {\name};
  }
\end{tikzpicture}
\end{center}

\medskip 

The length of this path for the metric \eqref{ds2} is 
$$
  L = 4A + 2p\,\ee^{-A}.
$$
It is minimised at $A=\log\bigl(\tfrac{p}{2}\bigr)$ (we henceforth assume $p\ge 2$), giving
$$
  L(p) = 4\log\bigl(\tfrac{p}{2}\bigr)+4 = 4 \log(p) + 4(1-\log(2)).
$$
Assuming $p$ large enough,  say $p\ge 6$, one has $L(p)<\sqrt{2}\,p$, and therefore 
$$
  d(O,P) < L(p) = 4 \log(p) + 1.23,
$$
which is much shorter than $\sqrt{2}\; p$.  Surprisingly, this crude construction  gives us the correct asymptotic estimates of the  true distance for large $p$, see Theorem \ref{th.largedistance}.

\subsection{The geodesic equation}

As mentioned in the introduction, the geodesic equations on  $\SOL$ form a completely integrable  system and can be solved by quadrature. Recall that a smooth curve $\gamma(t) = (x(t), y(t), z(t))$ in $\SOL$ is a geodesic if and only if it is a critical point of the action functional, which is defined as the integral
$$
 \mathcal{S}(\gamma) = \int \mathcal{L}(x, y, z, \dot{x}, \dot{y}, \dot{z}) \, dt,
$$
where the Lagrangian is given by 
$$
 \mathcal{L} = \mathcal{L}(x, y, z, \dot{x}, \dot{y}, \dot{z}) = \frac{1}{2}\|\dot{\gamma}(t)\|^2 =
 \frac{1}{2}\left(\mathrm{e}^{-2z}\dot{x}^2 + \mathrm{e}^{2z}\dot{y}^2 + \dot{z}^2\right).
$$

The associated Euler-Lagrange equations are
\begin{equation}\label{EulerLagrance}
  \frac{d}{dt}\frac{\partial \mathcal{L}}{\partial \dot{x}} = \frac{\partial \mathcal{L}}{\partial x}, \qquad
 \frac{d}{dt}\frac{\partial \mathcal{L}}{\partial \dot{y}} = \frac{\partial \mathcal{L}}{\partial y}, \qquad
 \frac{d}{dt}\frac{\partial \mathcal{L}}{\partial \dot{z}} = \frac{\partial \mathcal{L}}{\partial z}.
\end{equation}

Explicitly, these equations yield
\begin{equation}\label{eq.ELbis}
\frac{d}{dt}\left(\mathrm{e}^{-2z}\dot{x}\right) = 0, \quad \frac{d}{dt}\left(\mathrm{e}^{2z}\dot{y}\right) = 0, 
\quad \text{and} \quad \frac{d^2 z}{dt^2} = -\mathrm{e}^{-2z}\dot{x}^2 + \mathrm{e}^{2z}\dot{y}^2.
\end{equation}

\medskip

This system of ODEs decouples into three simpler equations:
\begin{lemma}\label{zpp}
A smooth curve $\gamma(t)=(x(t),y(t),z(t))$ is a geodesic in $\SOL$ if and only if there exist constants $a,b\in\mathbb{R}$ such that:
\begin{equation}\label{eq.xyp}
\dot{x}(t)=a\,\mathrm{e}^{2z(t)},\qquad \dot{y}(t)=b\,\mathrm{e}^{-2z(t)}.
\end{equation}
and
\begin{equation}\label{eq.zpp}
\ddot{z}(t)=-a^{2}\mathrm{e}^{2z(t)}+b^{2}\mathrm{e}^{-2z(t)}.
\end{equation}
\end{lemma}

\medskip 

We call $a$ and $b$ the \emph{principal constants of motion}; they are first integrals of the geodesic flow. 
It will be convenient to rewrite \eqref{eq.zpp} as
\begin{equation}\label{eq.zppp}
  \ddot{z}(t)=-\,U'_{a,b}\bigl(z(t)\bigr),
\end{equation}
where the \emph{potential function} is defined by
\begin{equation}\label{def.potential}
U_{a,b}(z)=\tfrac12\bigl(a^{2}\mathrm{e}^{2z}+b^{2}\mathrm{e}^{-2z}\bigr),
\end{equation}
and $U_{a,b}'(z)$ denotes the derivative with respect to $z$.

\begin{proof}[Proof of the Lemma]
From the first two equations in \eqref{eq.ELbis} we see that there exist constants $a,b\in\mathbb{R}$ such that
$$
  \dot{x}(t)=a\,\mathrm{e}^{2z(t)},\quad\dot{y}(t)=b\,\mathrm{e}^{-2z(t)},
$$
and \eqref{eq.xyp} follows. Equation \eqref{eq.zpp} is obtained by substituting these relations into the third  equation in \eqref{eq.ELbis}:
$$
\ddot{z}(t) =-\mathrm{e}^{-2z}\dot{x}^2+\mathrm{e}^{2z}\dot{y}^2=-a^2\mathrm{e}^{2z(t)}+b^2\mathrm{e}^{-2z(t)}.
$$
\end{proof}

\medskip

Since $\SOL$ is a homogeneous Riemannian manifold, it is complete and therefore all geodesics can be extended for all time $t \in \mathbb{R}$. Throughout the paper, we restrict our considerations  to geodesics of unit speed, for which the following relation holds:
$$
 \|\dot{\gamma}(t)\|^2 = \text{e}^{-2z}\dot{x}^2 + \text{e}^{2z}\dot{y}^2 + \dot{z}^2 = 1.
$$
Using the expressions for $\dot{x}$ and $\dot{y}$ in the previous Lemma, this identity reduces to
\begin{equation} \label{UnitSpeed}
  2U_{a,b}(z) + \dot{z}^2   =  a^2\text{e}^{2z} + b^2\text{e}^{-2z}  + \dot{z}^2 = 1.
\end{equation}
Observe that this normalization implies that  $2|ab| \leq  1$; indeed we have from \eqref{UnitSpeed}
\begin{equation} \label{boundab}
   2\, |ab| \leq    2U_{a,b}(z) = a^2\text{e}^{2z} + b^2\text{e}^{-2z}  \leq 1.
\end{equation}

\medskip 

The next lemma records two additional basic facts about geodesics in $\SOL$.

\begin{lemma}\label{basicrelations}
Let $\gamma(t)=(x(t),y(t),z(t))$ be a unit‐speed geodesic in $\SOL$. Then:
\begin{enumerate}[(i)]
\item The function 
  $$
    c  :=  a\,x(t)\;-\;b\,y(t)\;+\;\dot z(t)
  $$
  is constant along $\gamma$.
\item For every $t$ we have 
  $$
    \bigl(a\,x(t) - b\,y(t) - c\bigr)^{2}  +a^{2}\ee^{2z(t)}\;+\;b^{2}\ee^{-2z(t)}
     = 1.
  $$
\end{enumerate}
\end{lemma}

\begin{proof}
From \eqref{eq.xyp} and \eqref{eq.zpp} we have
$$
  \dot c = a\dot x - b\dot y + \ddot z  \equiv 0,
$$
which proves the first claim. The second claim follows from the first and the  identity~\eqref{UnitSpeed}.
\end{proof}

\medskip 

\begin{definition}\label{def.critical}
 A point on a (non constant) geodesic $\gamma$ is said to be a \emph{critical point} if \,  $\dot z(t) = 0$. It is an \emph{inflection point} if \,   $\ddot z(t) = 0$.
\end{definition}
 From \eqref{zpp}, we see that  the inflection points only exist if  both  $a\neq 0$  and  $b\neq 0$. They  are characterized by the condition 
\begin{equation*}
   U_{a,b}'(z) = a^2\mathrm{e}^{2z} - b^2\mathrm{e}^{-2z} = 0, 
\quad \text{equivalently} \quad z = h(a,b) :=  \frac{1}{2} \log \left| \tfrac{b}{a}\right|. 
\end{equation*}
From  Lemma \ref{basicrelations},   the critical points are characterized by
$$
  2U_{a,b}(z) = a^2 \ee^{2z} + b^2 \ee^{-2z} = 1,    \quad \text{equivalently} \quad    c = ax - by.
$$

\subsection{A classification of geodesics}

Two complete geodesics $\gamma_1$, $\gamma_2$ are said to be \emph{geometrically equivalent}\index{geodesic!geometrically equivalent}
 if there exists an isometry $\Phi$ of $\SOL$ such that $\gamma_2$ is a reparametrization of $\Phi \circ \gamma_1$. In this section we discuss how the constants $a$ and $b$ control the shape of a given, non constant, geodesic. If $a = b = 0$, then $x$ and $y$ are constant  and the geodesic is a vertical line. The other geodesics are classified in the following three types:
\begin{enumerate}[-]
\item A geodesic is \textit{horizontal}\index{geodesic!horizontal}  if $z(t)$ is constant.
\item A geodesic is \textit{hyperbolic}\index{geodesic!hyperbolic}  if it is contained in a leaf of one of the vertical foliations $\mathcal{H}'$ or $\mathcal{H}''$, that is,  if either $x(t)$ or $y(t)$ is constant.
\item The other geodesics are said to be \textit{generic}.\index{geodesic!generic}
\end{enumerate}
In  \cite{Troyanov98}, these geodesics are respectively called of type A, B or C. It is clear that a geodesic is vertical if and only if it has no critical point and  it is horizontal if and only if all of its points are critical.
Hyperbolic geodesics have exactly one critical point and no inflection point. Generic geodesics encounter periodically both inflection points and critical (maximal and minimal) points. 

\medskip

The following notion  will play a key role in our geometric understanding of the geodesics:
\begin{definition}\label{def.modulus}
 The \emph{modulus}\index{modulus} of a (non trivial)  geodesic  $\gamma(t) = (x(t), y(t), z(t))$ in $\SOL$ is  the parameter $k\in [0,1]$ defined  as 
$$
  k := \sqrt{\frac{\|\dot \gamma\|^2-2|\dot x \dot y|}{\|\dot \gamma\|^2+2|\dot x \dot y|}}. 
$$ 
\end{definition}
The modulus is  constant along the geodesic and invariant under reparametrizations. 
If the geodesic has unique speed, then  using \eqref{eq.xyp}, we have:
$$
  k = \sqrt{\frac{1-2|\dot x \dot y|}{1+2|\dot x \dot y|}} = \sqrt{\frac{1-2|ab|}{1+2|ab|}},
$$ 
where $a$ and $b$ are the principal constants of motion. In particular, the modulus is a constant along the geodesic.

\medskip 

Our first result states that the modulus of a geodesic  is invariant under the isometries of $\SOL$:

\begin{lemma}\label{kisinvariant1}
Two geometrically equivalent geodesics have the same modulus. 
\end{lemma}

\begin{proof}
We first observe that any isometry of $\SOL$ that belong to the isotropy group of the origin clearly does not affect the modulus of a geodesic. 
The  isometries coming from the left action of the group $\SOL$ have the form 
$$
\Phi_\lambda : (x, y, z) \mapsto ( \mathrm{e}^{\lambda} x + x_0, \mathrm{e}^{-\lambda} y+y_0, z + \lambda),
$$
for some $\lambda \in \R$; and a trivial calculation shows that the  geodesics
$\gamma$ and $\Phi_\lambda(\gamma)$ have the same modulus.
\end{proof}

\medskip

We will prove later in  Proposition \ref{kisinvariant2}  that the converse also holds:   \emph{two generic geodesics are geometrically equivalent if and only if they have the same modulus}.

\bigskip 

The  type of a geodesic can be described in terms of its modulus: 

\begin{lemma}\label{lem.3types}
Let $\gamma$ be a unit-speed geodesic in $\SOL$ with modulus $k\in[0,1]$. Then
\vspace{-5pt}
\begin{alignat*}{3}
& \gamma \text{ is vertical or hyperbolic} \quad &  \Leftrightarrow \quad & k=1,\\
& \gamma \text{ is generic}                    \quad &\Leftrightarrow\quad & 0<k<1,\\
& \gamma \text{ is horizontal}                 \quad &\Leftrightarrow\quad & k=0.
\end{alignat*}
In the last case, the height is constant and equals to 
$$
  z(t) =  h =  \tfrac{1}{2}\log \left|\frac{b}{a}\right|.
$$
\end{lemma}

\begin{proof}
Vertical and hyperbolic geodesics are characterized by the condition $\dot x \dot y  \equiv 0$, which is equivalent to $k = 1$. 
To distinguish between horizontal and generic geodesics, we use the inequality
$$
  2|ab| \leq a^2\text{e}^{2z} + b^2\text{e}^{-2z}.
$$
Combined with \eqref{UnitSpeed}, this implies
$$
 \dot{z}^2 = 1 - \left( a^2\text{e}^{2z} + b^2\text{e}^{-2z} \right) \leq 1 - 2|ab|.
$$
It follows that if  $k=0$, then  $2|ab| = 1$, therefore  $\dot{z} \equiv 0$ and the geodesic is horizontal.

\smallskip 

Conversely, if $z$ is constant and $ab\neq0$, then equation \eqref{eq.zpp} yields
$$
 0 = \ddot{z} = -a^2\text{e}^{2z} + b^2\text{e}^{-2z} \quad \Rightarrow \quad \ee^{4z} = \left( \frac{b}{a} \right)^2,
$$
which gives
$$
 z = \tfrac{1}{2} \log \left| \frac{b}{a} \right|.
$$
Substituting into \eqref{UnitSpeed} gives us 
$$
 1 = a^2\text{e}^{2z} + b^2\text{e}^{-2z} =  2|ab|,
$$
therefore  $2|ab| = 1$,  that is, $k = 0$. The proof of the Lemma is complete.
\end{proof}

 \newpage
 
\subsection{Special geodesics} \label{special}

In this section we completely describe all the special geodesics. 

\subsubsection{Vertical and horizontal geodesics}

A trivial case occurs when $a=b=0$. Then $\dot x=\dot y=0$ and $\ddot z=0$, hence $x$ and $y$ are constant while $z$ is an affine function of $t$; these are the vertical geodesics.

\medskip

Another simple case arises when $ab\neq 0$ and, for some $t_{0}$, one has
\[
\dot z(t_{0})=0\quad\text{and}\quad z(t_{0})=h=\tfrac12\log\!\left|\frac{b}{a}\right|.
\]
At this point the potential attains its minimum, $U_{a,b}(h)=|ab|$. Using \eqref{UnitSpeed}, we  see that 
$2 |ab| = 2U_{a,b}(h) = 1$, hence $2|ab| = 1$, which implies that the modulus $k = 0$. By Lemma \ref{lem.3types},  the geodesic is horizontal at  height $h$.
Moreover,
\[
\dot x = a\,\mathrm{e}^{2h}=\operatorname{sgn}(a)\,|b|=\pm b,
\qquad
\dot y = b\,\mathrm{e}^{2h}=\operatorname{sgn}(b)\,|a|=\pm a,
\]
so the curve is a horizontal straight line of the form
\[
 \gamma(t) = (x_{0}\pm b\,t,\; y_{0}\pm a\,t,\; h).
\]
In the moving frame \eqref{FrameXYZ}, we have 
$$
 \dot\gamma(t) = \frac{1}{\sqrt{2}}\left(\pm X \; \pm Y\right),
$$
hence the  horizontal geodesic  $\gamma$ meets each vertical foliation at angle $\pi/4$.

\subsubsection{Hyperbolic geodesics.} \label{sec.HyperbolicGeodesics}

We now consider the case $b = 0$ and $a \ne 0$, so that $\dot{y} = b\, \mathrm{e}^{-2z} \equiv 0$. Hence $y(t)$ is constant and can be ignored. Without loss of generality, we assume $a>0$, and normalize the speed to  $\|\dot{\gamma}(t)\| = 1$.

\medskip 

In that case, describing the geodesics is almost trivial. 
Indeed, we know from Lemma \ref{basicrelations} that $c = \dot z +ax$ is constant and that 
$$
  (ax  - c)^2 + a^2 \ee^{2z}  = 1.
$$
This relation completely describes all the possible hyperbolic geodesics (with $b=0$).  

Introducing the  coordinates $ u = x $ and $ v = \mathrm{e}^{z} $, we  rewrite this as
$$
 (u - u_0)^2 + v^2 = \frac{1}{a^2},
$$
where $u_0 = c/a$. This  is the equation of a Euclidean semicircle centered along the   axis $\{v=0\}$. This is consistent with the classical model of geodesics in the Poincaré upper half-plane, as expected.

 \medskip

So far, we have obtained the relation between the coordinates $x$ and $z$ satisfied by a  hyperbolic geodesic. To express these variables as function of the time $t$, we recall that  from \eqref{UnitSpeed} we have
$$
  \dot z^2 + a^2 \mathrm{e}^{2z}  =  1,
$$
which implies that $z(t)\le -\log (a)$.  A direct calculation shows that the solution to this equation with initial value $z(0) = z_{\mathrm{max}} = -\log (a)$
(so $\dot z (0) = 0$)   is given by 
$$
 z(t) = -\log \left(a \cosh t\right).
$$
Using now  \eqref{eq.xyp}, we have
$$
   \frac{dx}{dt} = a\,\mathrm{e}^{2z} =  \frac{1}{a\cosh^2(t)},
$$
and therefore 
$$
x(t) = x_0 + \int \frac{dt}{a \cosh^2 t} 
= x_0 + \frac{1}{a} \tanh(t),
$$
where $x_0$ is a constant of integration. 

\medskip

The unique  geodesic segment connecting the  two points $ (0, 0, 0) $ and $ (\lambda, 0, 0) $ in $\SOL$ is the curve
$$
 \beta : \left[ -\arccosh\left( \tfrac{1}{a} \right),\ \arccosh\left( \tfrac{1}{a} \right)\right]  \to \SOL
$$
defined by
$$
  \beta(t) = \left(\tfrac{\lambda}{2}+\tfrac{1}{a}\cdot \tanh(t), \, 0 , \ -\log\left( a \cosh(t) \right)\right),
$$
where  $a = \frac{2}{\sqrt{ \lambda^2+4 }}$. 
The distance between the endpoints is therefore
\[
  d  =  2\arccosh\!\left(\tfrac{1}{a}\right)   =  2\arccosh\!\left(\tfrac{\sqrt{\lambda^{2}+4}}{2}\right)
         =  2\operatorname{arsinh}\!\left(\tfrac{\lambda}{2}\right).
\]
We have asymptotically 
$$
 d = 2\log\lambda \;+\; O(1/\lambda^{2}).
$$
as $ \lambda \to \infty$.

\medskip

If $a=0$ and $b>0$, the roles of $x$ and $y$ are interchanged. Writing $c=\dot z - b\,y$ (which is constant), we obtain
\[
(by+c)^2 + b^2 \ee^{-2z} = 1.
\]
Assuming $\dot z(0)=0$ gives $c=-b\,y_0$, and the same analysis yields
\[
y(t)=y_0+\frac{1}{b}\tanh t,\qquad z(t)=\log\bigl(b\cosh t\bigr).
\]

\section{The generic geodesics}\label{sec:generic} 

For a generic unit-speed geodesic $\gamma(t)=(x(t),y(t),z(t))$ in $\SOL$, with nonzero constants of motion $a,b$, we define
the \textit{average height} of $\gamma$ as  
\[
  h = h(a,b)  =  \frac{1}{2}\log \left|\frac{b}{a}\right|.
\]
The parameters $a,b,h$ and the modulus $k$ are related by the identities
\begingroup
\setlength{\fboxsep}{6pt}
\setlength{\fboxrule}{0.3pt}
\renewcommand{\arraystretch}{1.2}
\begin{equation}\label{eq:hkab-summary}
\boxed{%
\begin{array}{@{} r@{\,=\,}l @{\qquad\quad} r@{\,=\,}l @{}}
k & \sqrt{\dfrac{1-2|ab|}{\,1+2|ab|\,}} &
2|ab| & \dfrac{1-k^{2}}{1+k^{2}} 
\\[6mm]
h & \tfrac{1}{2}\log \left|\dfrac{b}{a}\right|  &
 \mathrm{e}^{2h} & \left|\dfrac{b}{a}\right| 
\\[6mm]
|a| & \mathrm{e}^{-h}\sqrt{\dfrac{1-k^{2}}{2(1+k^{2})}} &
|b| & \mathrm{e}^{h}\sqrt{\dfrac{1-k^{2}}{2(1+k^{2})}}
\end{array}}
\end{equation}
\endgroup

\medskip

\begin{remark}
\begin{enumerate}[--]\rm
\item The signs of $a$ and $b$ are not determined by $k$ and $h$. They can be adjusted by isometries (reflections), so one may assume $a,b>0$ if desired.
\item We will often adopt the following convention:  \textit{If a function $f$ depends  on $|ab|$, we simply write $f(k)$ without
systematically substituting  $|ab|$ by  its value \, $\frac{1-k^{2}}{2(1+k^{2})}$.}
\end{enumerate}
\end{remark}

\subsection{The vertical periodicity in generic geodesics}\label{sec.discussZ}
With the above notation, the potential can be written as
\begin{equation}\label{eq:Ucosh}
  U_{a,b}(z)=\tfrac12\bigl(a^{2}\mathrm{e}^{2z}+b^{2}\mathrm{e}^{-2z}\bigr)
  =|ab|\,\cosh\bigl(2(z-h)\bigr).
\end{equation}
In particular, \eqref{eq.zpp} takes the form\footnote{This is the analog of the classical pendulum equation $\ddot{\theta}=-\omega^{2}\sin\theta$, with $\sin$ replaced by $\sinh$. The solutions share the same qualitative features: each trajectory oscillates periodically about a stable equilibrium.}
\begin{equation}\label{eq.zppbis}
    \frac{d^{2}}{dt^{2}}\bigl(z-h\bigr) = -\frac{d}{dz}U_{a,b}(z)
    = -\,2|ab|\,\sinh\bigl(2(z-h)\bigr).
\end{equation}

To construct a generic unit-speed geodesic $\gamma(t)=(x(t),y(t),z(t))$ in $\SOL$, we first solve the scalar ODE \eqref{eq.zppbis} for $z(t)$. The horizontal coordinates then follow by direct integration from
\[
  \dot{x}=a\,\mathrm{e}^{-2z},\qquad \dot{y}=b\,\mathrm{e}^{2z}.
\]

Since $ab\neq 0$ by hypothesis, the potential $U_{a,b}$ is smooth, strictly convex, and attains its unique minimum
\[
  \min U_{a,b}(z)=|ab|\quad\text{at}\quad z=h=\tfrac12\log\Bigl|\tfrac{b}{a}\Bigr|.
\]

 From \eqref{UnitSpeed} we have: 
 \begin{equation}\label{Lagrangian3}
 	\dot z^2 + 2U_{ab}(z) = \dot z^2 +  2|ab| \cosh(2(z-h)) = 1.
 \end{equation}
 In particular we have $|\dot z| \leq \sqrt{1-2|ab| }$,  and 
$$
 z_{\min} =   h - A \  \leq  \ z(t)  \ \leq h + A = z_{\max},
$$
where the \textit{amplitude}\index{geodesic!amplitude} is defined by $A > 0$ and $2 U_{a,b}(h+A) = 1$, that is,
\begin{equation}\label{eq.amplitude}
 A = A(k) =  \frac{1}{2}\arccosh\left( \frac{1}{2 |ab|}\right)  =   \mathrm{arctanh} (k).
\end{equation}
Equation \eqref{Lagrangian3} describes a simple closed convex curve $\mathcal{C}_{a,b}$ in the $(z,\dot z)$-plane (the \emph{phase curve}). Any solution of \eqref{eq.zppbis} yields a smooth parametrization of $\mathcal{C}_{a,b}$. This curve is invariant under the symmetries $\dot z\mapsto -\dot z$ and $z\mapsto 2h-z$.

\begin{center}
\begin{tikzpicture}[scale=0.81]
  \def\h{1}        
  \def\ab{0.15}    
  \begin{axis}[
    axis lines       = middle,
    axis line style  = {-stealth},
    xlabel           = {$z$},
    ylabel           = {$\dot z$},
    xlabel style     = {below},
    ylabel style     = {right},
    ticks            = none,
    axis equal image,
    enlarge x limits = 0.1,
    enlarge y limits = 0.2,
    width  = 8cm,
    height = 6cm,
  ]
  \addplot[dashed] coordinates {(\h,-0.05) (\h,0.05)};
  \node[below] at (axis cs:\h,0) {$h$};
  \addplot[
    thick,
    domain=-1.9:1.9,
    samples=200,
    unbounded coords=jump,
    postaction={
      decorate,
      decoration={
        markings,
        mark=at position 0.72 with {\arrow[scale=1.2]{>}}
      }
    },
  ]
  ({x+\h}, {sqrt(max(0,1 - 2*(\ab)^2*cosh(2*x)))});
  \addplot[
    thick,
    domain=-1.9:1.9,
    samples=200,
    unbounded coords=jump,
  ]
  ({x+\h}, {-sqrt(max(0,1 - 2*(\ab)^2*cosh(2*x)))});
  \end{axis}
\end{tikzpicture}
\end{center}

\bigskip

Any solution $z(t)$ of \eqref{eq.zppbis} is  periodic, with period
\begin{equation}\label{PeriodT0}
   T = 2 \int_{z_{\min}}^{z_{\max}}   \frac{dz}{\sqrt{1 - 2U_{a,b}(z)}}
\end{equation}
To see why this is true, we write \eqref{Lagrangian3} as
$$
  \left| \frac{dz}{dt}\right| = \sqrt{1 - 2U_{a,b}(z)}.
$$
When $z(t)$ covers one half‐oscillation from $z_{\min}$ to $z_{\max}$,  we  have $\dfrac{dz}{dt} \geq 0$ and the above equation 
can be written in the following differential form:  
$$
  dt = \frac{dz}{\sqrt{1 - 2U_{a,b}(z)}}.
$$
Integrating this relation gives us \eqref{PeriodT0}. See \cite{Arnold1, Arnold2} for more details. 

\medskip

\begin{lemma}
The period only depends on the modulus $k$. More precisely, we have
\begin{equation}\label{PeriodT_general}
   T =      T(k)=4\!\int_{0}^{A(k)}
        \frac{d\zeta}{\sqrt{1-
        \bigl(\tfrac{1-k^{2}}{1+k^{2}}\bigr)\cosh(2\zeta)}}.
\end{equation}

\end{lemma}

\begin{proof}
Introducing the notation $\zeta =z-h$, we  have  $ - A(k)\, \leq \, \zeta \leq  A(k)$ and 
$$
 2U(z)=  2|ab| \cosh(2\zeta).
$$
Applying \eqref{PeriodT0}  gives us 
$$
 T = 4  \int_{0}^{A(k)}   \frac{d\zeta}{\sqrt{1 - 2U(h + \zeta)}} =  4 \int_{0}^{A(k)} \frac{d\zeta }{\sqrt{1 -  2|ab|\,\cosh\bigl(\zeta\bigr)}}.
$$
\end{proof}

\medskip
 
Basic properties of the period $T(k)$ are collected in Sections \ref{sec.TMHelliptic} and \ref{sec.TMHproperties} at the end of the paper, 
including its representation in terms of the complete elliptic integral of the first kind $K(k)$. The  integration of $x(t)$ and $y(t)$ is discussed in Section~\ref{sec.winding} below.

\subsection{The Grayson cylinder}\label{sec:grayson}

To each generic geodesic $\gamma$ in $\SOL$ we associate a companion surface that conveniently encodes some key features of the geodesic flow.

\begin{definition}\label{def.grayson}\rm
Given a unit-speed generic geodesic $ \gamma(t) $ in $ \SOL$ with constants of motion $a,b$,  
we define its \emph{Grayson cylinder}\index{Grayson cylinder}  as the level set
\[
  \mathcal{G}=\mathcal{G}_{a,b,c}  := \bigl\{(x,y,z)\in\SOL \;\bigm|\; (a x - b y - c)^{2} + 2\,U_{a,b}(z)=1 \bigr\} \subset \SOL, 
\]
where $c := a\,x(t) - b\,y(t) + \dot z(t)$ is the constant from Lemma~\ref{basicrelations}\,(i). 

\medskip 

The \emph{modulus} of $\mathcal{G}$ is by definition the modulus of the underlying geodesic:
\[
  k = \sqrt{\tfrac{1-2|ab|}{1+2|ab|}}.
\]
In particular, $k$ depends only on the cylinder $\mathcal{G}$ (all geodesics it contains have the same modulus).
\end{definition}

\medskip

This surface is smooth and diffeomorphic to $S^{1}\!\times\!\mathbb{R}$; its relevance is captured by the following

\begin{proposition}
A generic geodesic $\gamma$ is entirely contained in its associated Grayson cylinder $\mathcal{G}_{a,b,c}$.
\end{proposition}

\begin{proof}
Immediate from Lemma~\ref{basicrelations}.
\end{proof}

This observation and its importance for understanding geodesics in $\SOL$, was first made and emphasized by M.~A.~Grayson~\cite{Grayson}\footnote{More generally, the existence of such invariant cylinders reflects the complete (super) integrability in the Liouville sense of the geodesic flow, see \cite[\S 49, 50]{Arnold2}.} in his thesis. We borrow the terminology from  \cite{CoiculescuSchwartz}.

\medskip

The following lemma describes useful bounds containing the Grayson cylinder:
\begin{lemma}\label{boundsGrayson}
The Grayson cylinder $\mathcal{G}_{a,b,c}$ is contained in the region defined by 
\[
  |a x - b y - c| \le \sqrt{1 - 2|ab|} 
  \quad  \mathrm{and} \quad 
  |z- h| \le A(k),
\]
where
$A(k)$ is the amplitude defined in \eqref{eq.amplitude}.
\end{lemma}

\begin{proof}
Since  $U_{a,b}(z)\ge |ab|$, we have
$$
  (a x - b y - c)^{2} = 1 - 2\,U_{a,b}(z) \le 1 - 2|ab|,
$$
the first inequality is then obvious from the definition of $\mathcal{G}_{a,b,c}$. The second inequality follows from the  equivalence 
$$
   2U_{a,b}(z) \leq 1  \quad \Leftrightarrow \quad       |z-  h|  \leq   A(k). 
$$
\end{proof}

\medskip

We now introduce some  terminology:
\begin{definition}\rm
\begin{enumerate}[(1)]
\item The horizontal plane defined by $z=h=\tfrac{1}{2}\log\!\left|\tfrac{b}{a}\right|$ is called the \emph{equatorial plane} of $\mathcal{G}$.
\item The horizontal line
\[
  z=h,\qquad a x - b y = c,
\]
is the \emph{axis} of $\mathcal{G}$; it lies in the equatorial plane and is disjoint from $\mathcal{G}$.
\item The equatorial plane intersects $\mathcal{G}$ along the two horizontal lines
\[
  z=h,\qquad a x - b y = c \pm \sqrt{1 - 2|ab|} = c \pm \sqrt{ \tfrac{2k^2}{1+k^{2}}}
\]
We call them the \emph{inflection lines}. They consist of the inflection points of the geodesics (i.e. $\ddot z=0$).
They are the point where $|\dot z|$ attains its maximal value. 
The inflection lines are  horizontal geodesics.
\item The horizontal lines
\[
  a x - b y = c,\qquad    z - h = \pm A(k)  
\]
are the \emph{critical lines}: they consist of the critical points of the geodesics contained in $\mathcal{G}$  (i.e. $\dot z=0$). These lines are \emph{not} geodesics of $\SOL$.
\end{enumerate}
\end{definition}

We conclude with two simple observations: 

\begin{remark}\rm
\begin{enumerate}[(i)]
\item $\mathcal{G}$ contains  the origin $0$ if and only if  $a^{2}+b^{2}+c^{2}=1$.
\item The axis of $\mathcal{G}$  passes through $0$ if and only if $c = 0$ and $a=b$ (so $h=0$).
\item Any isometry of $\SOL$ that preserves the axis also leaves $\mathcal{G}$ invariant.
\end{enumerate}
\end{remark}

\subsection{An application: the modulus as a complete invariant}

As a first application of the properties of the Grayson cylinder, we prove the following result, which completes Lemma~\ref{kisinvariant1}:

\begin{proposition}\label{kisinvariant2}
Two generic geodesics in $\mathrm{SOL}$ are geometrically equivalent if and only if they have the same modulus; equivalently, they have the same amplitude $A(k)=\operatorname{arctanh}(k)$. 
\end{proposition}

Recall that two complete geodesics are \emph{geometrically equivalent} if they differ by an isometry of $\SOL$ and/or a reparametrization.

\begin{proof}
Let $\gamma$ be a generic geodesic with modulus $k$ and Grayson cylinder $\mathcal{G}$.
We perform the following normalizations using isometries and a reparameterization:

\smallskip\quad
(i) Move $\mathcal{G}$ by a vertical lift so that its equatorial plane becomes the ground plane $\Pi_0=\{z=0\}$; equivalently, $h=0$, hence $|a|=|b|$.

\smallskip\quad
(ii) Horizontally translate $\mathcal{G}$  so that its axis  passes through the origin. We then have $c=0$.

\smallskip\quad
(iii) If necessary, reflect in the coordinate planes to arrange $a=b>0$.

\smallskip\quad
(iv) Normalize the speed and the initial phase so that $\|\dot{\gamma}\|=1$ and $\gamma$ starts at a critical point:
\[
z(0)=z_{\max}=A(k),   \qquad \dot z(0)=0.
\]

\smallskip\quad
(v) Translate   $\mathcal{G}$ {along its axis} if necessary,  so that $x(0) = y(0) = 0$. 

\smallskip
After these normalizations we have
\[
\gamma(0)=(0,0,A(k)),\qquad
\dot\gamma(0)=\bigl(a\,\ee^{2A(k)},\,a\,\ee^{-2A(k)},\,0\bigr),
\]
with $a=b>0$ uniquely determined by
\[
2a^2=\frac{1-k^2}{1+k^2}.
\]
The normalized model geodesic is uniquely determined by its modulus $k$, thus two generic geodesics with the same modulus are geometrically equivalent. The converse implication is Lemma~\ref{kisinvariant1}.
\end{proof}

\medskip

\subsection{Framing the Grayson cylinders}

Given $a,b,c \in \R$ such that  $0<2 |ab| < 1$, we introduce the following global vector fields $\xi$ and $\eta$ on $\SOL$ defined by 
\begin{eqnarray*}
\xi  &=& \xi_{a,b,c}  = a \ee^{2z}\,\frac{\partial}{\partial x} + b \ee^{-2z}\,\frac{\partial}{\partial y}
      + (c-ax+by)\,\frac{\partial}{\partial z} \\[1ex]    
&=&   a \ee^{z} X + b \ee^{-z} Y + (c-ax+by) Z, \\[2ex]      
\eta &=& \eta_{a,b}   =  b\,\frac{\partial}{\partial x}  + a\,\frac{\partial}{\partial y}   = b \ee^{-z} X + a \ee^{z} Y,
\end{eqnarray*}
where $ \{X, Y, Z\} $ is the left-invariant orthonormal frame defined in~\eqref{FrameXYZ}.

\medskip

We collect the basic properties of these vector fields : 

\begin{proposition}\label{prop:xi-eta}
\begin{enumerate}[(i)]
\item For every point of the Grayson cylinder $\mathcal G_{a,b,c}$ the vectors
      $\xi$ and $\eta$ are tangent to~$\mathcal G_{a,b,c}$.
\item Any integral curve of $\xi$ through a point in $\mathcal G_{a,b,c}$ is a geodesic
      of $\SOL$. In particular it is globally contained in  $\mathcal G_{a,b,c}$.
\item The integral curves of $\eta$ are horizontal lines contained  in~$\mathcal G_{a,b,c}$ and  parallel to the cylinder’s axis.
      They are geodesic precisely along the two inflection lines.
\item The flows generated by $\xi$ and $\eta$ commute;  both leave the Grayson cylinder invariant.
\item At any point of the Grayson cylinder we have 
$$
     2|ab|\ \le\ |\cos(\theta)|\ \le\ \sqrt{2|ab|}\,<1,
$$
where $\theta$ is the Riemannian angle between $\xi$ and $\eta$. In particular $\theta$ is uniformly separated from $0$, $\pi/2$, and $\pi$ by  constants that depend only on $a,b$.
\item The sign of $\cos\theta$ is the sign of $ab$.  Its absolute value attains its minimum along the critical lines  $a x-b y=c$ and its maximum along the inflection lines $z=h$ of $\mathcal G_{a,b,c}$.
\item Any isometry of $\SOL$ fixing the cylinder’s  axis  and its direction preserves both $\mathcal G$ and the vector fields $\xi,\eta$.
\end{enumerate}
\end{proposition}

\medskip 

\begin{proof}
{(i)}\;
The tangent space of $\mathcal G_{a,b,c}$ at $(x,y,z)$ is the kernel of the
$1$-form
\[
  \omega  =  (a x-b y-c)\,(a\,dx-b\,dy)  \;+\;\bigl(a^{2}\ee^{2z}-b^{2}\ee^{-2z}\bigr)\,dz
\]
(this is the differential of the defining function  of $\mathcal G_{a,b,c}$). The vector fields $\xi, \eta$ satisfy  $\omega(\xi)=\omega(\eta)=0$.

\medskip 

{(ii)}\;
By Lemmas~\ref{zpp} and~\ref{basicrelations}, the equations of a geodesic with constants of motion $a,b,c$  coincide with the
components of~$\xi$, hence each integral curve of~$\xi$ is a geodesic.

\medskip 

{(iii)}\;
The integral curve of $\eta$ through a point $(x_0,y_0,z_0) \in \mathcal{G}$ is the horizontal straight $t \;\longmapsto\; (x_{0}+bt,\; y_{0}+at,\; z_{0})$. It is clearly contained in $\mathcal{G}$ and parallel to the axis.
By Lemma \ref{lem.3types},  this line is a geodesic of $\SOL$ if and only if $z = h$, that is, $a\,\ee^{2z}=b\,\ee^{-2z}$. 
These two lines are the  inflection lines
$$
  z = h,  \qquad   a x - b y = c \pm \sqrt{1-2|ab|}.
$$
For every other height $z_{0}\neq h$, the horizontal line is  contained in $\mathcal G_{a,b,c}$ but is \emph{not} a geodesic.

\medskip 

{(iv)}\;
The flows of $\xi$ and $\eta$ commute because  $[\xi,\eta]= 0$. Both flows  preserve the  Grayson cylinder,  by statements (ii) and (iii). 

\medskip 

{(v)\,--\,(vi)}\;
On  $\mathcal{G}_{a,b,c}$ we have the Riemannian norms \  $\|\eta\| = \sqrt{1 - (ax - by - c)^2}$ and  $\|\xi\| = 1$.
The angle $\theta$ between $\xi$ and $\eta$ satisfies then 
$$
\cos(\theta)  = \frac{\langle\xi,\eta\rangle}{\|\xi\|\|\eta\|} = \frac{2ab}{\sqrt{1 - (ax - by - c)^2}}.
$$
This shows that the sign of $\cos\theta$ is that of $ab$. Using  Lemma \ref{boundsGrayson}, we   check the inequalities 
$$
  \sqrt{2|ab| }\leq  \|\eta\| \leq 1,
$$
therefore 
$$
   0<  {2|ab| }\leq |\cos(\theta)|\leq \sqrt{2|ab|}.
$$
 Equality on the left occurs when $\|\eta\|=1$, i.e. along the critical lines $a x-b y=c$. Equality on the right occurs when $\|\eta\|=\sqrt{2|ab|}$, i.e. along the inflection lines $z=h$.
 
 \medskip
 
The last statement (vii) is immediate from the definitions of $\mathcal{G} _{a,b,c}$ and the fields $\xi$, $ \eta$.
\end{proof}

\medskip

We conclude this section with a comment on time-reversal symmetry: 

\begin{remark} \rm 
\begin{enumerate}[(i)]
\item Given a generic  geodesic $\gamma$, the symmetric geodesic $\gamma^-$, defined by  $\gamma^{-}(t):=\gamma(-t)$, generate the same Grayson cylinder. Indeed,  the constants of motions $a,b,c$ of  $\gamma^{-}(t)$ coincide with those of $\gamma$ up to sign and clearly 
$$
 \mathcal{G}_{-a,-b,-c}  =\mathcal{G}_{a,b,c}.
$$
The corresponding vector fields then coincide up to sign: 
$$
  \xi^-  = -\xi,  \qquad \eta^-  =-\eta.
$$
These relations reflect the  time reversal $t\mapsto -t$ on the geodesic.
\end{enumerate}
\end{remark}

\subsection{Winding around the axis}\label{sec.winding}

We continue the discussion of geodesics from Section~\ref{sec.discussZ}. In this section we prove that a
generic geodesic winds around its associated Grayson cylinder with a constant horizontal drift after each
period. The following result is essentially due to Grayson, although not stated in this form (see pages 68-75, and the figure page 84, in \cite{Grayson}).

\begin{theorem}\label{th:spiraling}
Let \(\gamma(t)=(x(t),y(t),z(t))\) be a generic unit-speed geodesic with nonzero constants of motion $a,b$. 
Denote by $h$ its average height and by $k$ its modulus, and define
\begin{equation}\label{defMk}
   M(k) =  4\int_{0}^{A(k)} \frac{\cosh(2\zeta)\,d\zeta} {\sqrt{1-2|ab|\,\cosh(2\zeta)}},
\end{equation}
where $A(k)$ is the amplitude. 
Then 
\begin{enumerate}[(A)]
\item For every \(t_{0}\in\R\) we have $z(t_{0}+T(k))=z(t_{0})$ and 
\begin{equation}\label{eq.sw1}
   \gamma(t_{0}+T(k))-\gamma(t_{0}) = \sgn(ab) M(k)\,(b,\, a,\, 0).
\end{equation}
\item 
The horizontal distance between these points is equal to 
\begin{equation}\label{eq.sw2}
  \delta\bigl(\gamma(t_{0}),\gamma(t_{0}+T(k))\bigr) =  M(k)\,\sqrt{2|ab|\,\cosh\bigl(2(z(t_{0})-h)\bigr)}.
\end{equation}
\end{enumerate}
\end{theorem}

Recall that the \emph{horizontal distance} between two points at the same altitude is defined in 
Remark \ref{remarkHorizontal}. 

\medskip

\begin{figure}[htbp]
\captionsetup[figure]{labelformat=empty}
  \centering
  \includegraphics[width=1.2\linewidth, trim=1.2cm 2.1cm 1.5cm 4.61cm, clip]{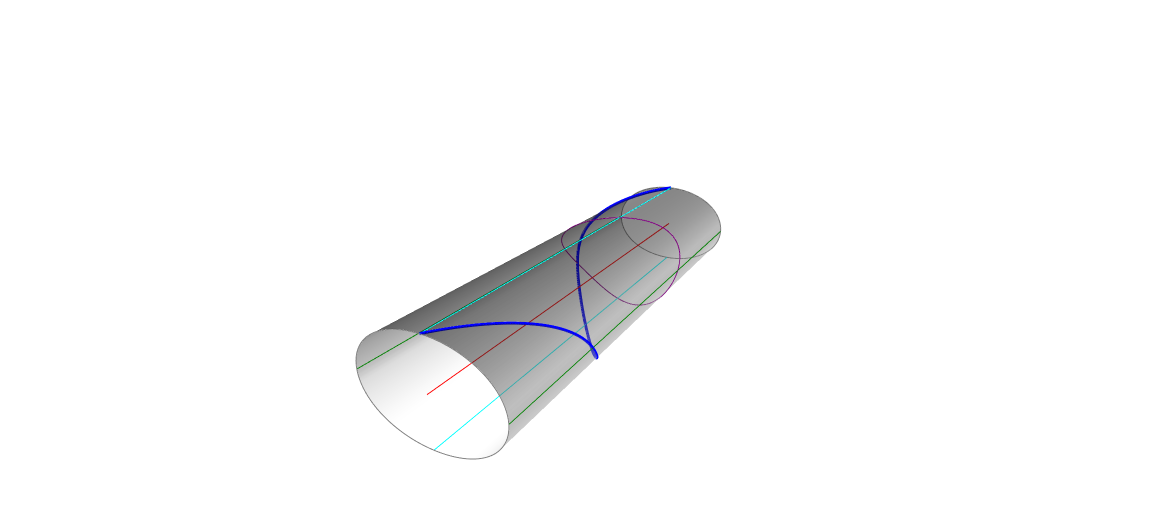}
  \caption*{A generic geodesic winding once around its Grayson Cylinder}
  \label{fig:partners}
\end{figure}

\begin{corollary}\label{HDI}
For a generic unit-speed geodesic $\gamma$,  the quantity $H$ defined by 
\begin{equation}\label{def.H}
   H :=  \sqrt{|x(t_0+T(k)) - x(t_0)|\;|y(t_0+T(k)) - y(t_0)|},
\end{equation}
is independent of $t_0$. More precisely, it only depends on the modulus of the geodesic and is given  by
\begin{equation}\label{defHDI}
  H = H(k)  = \sqrt{|ab|}\,  M(k).
\end{equation}
\end{corollary}

\medskip 

We call $H(k)$ defined by \eqref{defHDI}  the \emph{horizontal drift invariant} of $\gamma$. It is called the  \textit{holonomy invariant} by Coiculescu and Schwartz in \cite[page 2107]{CoiculescuSchwartz} and we borrow  here the notation $H$ from these authors. It is denoted  $D$ in \cite{Grayson} and  $p(k)$ in \cite{Troyanov98}.

\begin{remark} \rm 
(1) Equation \eqref{eq.sw2} can also be written as 
$$
  \delta\bigl(\gamma(t_{0}),\gamma(t_{0}+T)\bigr) =  
  M(k)\,\sqrt{\,a^{2}\ee^{2z(t_{0})}+b^{2}\ee^{-2z(t_{0})}\,}.
$$

(2) Since \(\cosh\bigl(2(z-h)\bigr)\ge 1\) with equality if and only if \(z=h\), we get
\[
  \delta\bigl(\gamma(t_{0}),\gamma(t_{0}+T)\bigr)\ \ge\ \sqrt{2|ab|}\,M(k),
\]
with equality precisely when \(z(t_0)=h\).

\smallskip 

(3) Although $H$ and $M$ are tightly linked by $H(k)=\sqrt{|ab|}\,M(k)$, they play different roles, and
keeping both functions is convenient. In particular, $M(k)$ expresses the horizontal drift in the direction of the axis of a generic geodesic  {(see \eqref{eq.sw1})}, while $H(k)$ is  more of a scalar invariant encoding the geometric mean of the coordinate increments in \eqref{def.H}.

\smallskip 

(4) In Sections \ref{sec.TMHelliptic} we derive explicit formulas for $M(k)$ and $H(k)$ in terms of complete elliptic integrals and in Section \ref{sec.TMHproperties} we give some important properties of these functions. 
\end{remark}

\medskip

For the proof of the theorem, we will need the following lemma, which we will prove later. 
\begin{lemma}\label{LemM}
For any \(t_{0}\in\R\), we have 
\[
\int_{t_{0}}^{t_{0}+T} \ee^{2z(t)}\,dt  =  \Bigl|\frac{b}{a}\Bigr|\,M(k),
\quad
\int_{t_{0}}^{t_{0}+T} \ee^{-2z(t)}\,dt  =  \Bigl|\frac{a}{b}\Bigr|\,M(k).
\]
\end{lemma}

\medskip

\begin{proof}[Proof of Theorem~\ref{th:spiraling}]
Since \(z(t)\) solves \eqref{eq.zppbis}, it is \(T\)-periodic. Moreover,
\[
\dot x = a\,\ee^{2z},\qquad \dot y = b\,\ee^{-2z},
\]
are \(T\)-periodic and uniformly bounded away from \(0\) and \(\infty\), with
\begin{equation}\label{boundsondotx}
\left|\tfrac{\dot x}{b}\right|,\ \left|\tfrac{\dot y}{a}\right|
\in \bigl[\ee^{-2A(k)},\,\ee^{2A(k)}\bigr]
= \bigl[\tfrac{1-k}{1+k},\,\tfrac{1+k}{1-k}\bigr].
\end{equation}
Thus both coordinates $x(t)$ and $y(t)$ decompose  as the sum of a linear term and a $T$-periodic function. It follows that $\gamma$ wraps around  the cylinder once every period while drifting horizontally at a constant rate.
For any \(t_{0}\), Lemma~\ref{LemM} gives
\[
\begin{aligned}
x(t_{0}+T)-x(t_{0})
&= a\int_{t_{0}}^{t_{0}+T}\ee^{2z(t)}\,dt
 = a\Bigl|\frac{b}{a}\Bigr|\,M  = \operatorname{sgn}(ab)\,b\,M,\\
y(t_{0}+T)-y(t_{0})
&= b\int_{t_{0}}^{t_{0}+T}\ee^{-2z(t)}\,dt
 = b\Bigl|\frac{a}{b}\Bigr|\,M
 = \operatorname{sgn}(ab)\,a\,M,
\end{aligned}
\]
which yields
$$
 \gamma(t_{0}+T)-\gamma(t_{0}) =\sgn(ab)\,M\,(b,a,0).
$$

\smallskip

Evaluating the horizontal distance at the common level \(z=z(t_{0})\) gives
$$
 \delta\bigl(\gamma(t_{0}),\gamma(t_{0}+T)\bigr) =   M(k) \sqrt{a^{2}\ee^{2z(t_{0})}+b^{2}\ee^{-2z(t_{0})} }
$$
and we conclude from $  a^{2}\ee^{2z(t_{0})}+b^{2}\ee^{-2z(t_{0})} = {2|ab|\,\cosh\bigl(2(z(t_{0})-h)\bigr)}$.
\end{proof}

\medskip 

\begin{proof}[Proof of Corollary \ref{HDI}]
From \eqref{eq.sw1} we have 
$$
  |x(t_0+T(k)) - x(t_0)|\;|y(t_0+T(k)) - y(t_0)| = M(k) \,  |b\, a| = H(k)^2.
$$
\end{proof}

\medskip 

We now prove Lemma \ref{LemM}:
\begin{proof}[Proof of the Lemma \ref{LemM}.]
Let us set $\zeta = z-h$. Then 
$$
  U_{a,b}(z)   =   |ab| \cosh(2(z-h))  = |ab| \cosh(2\zeta),  \quad \text{and} \quad  \ee^{2z} = \left|\frac{b}{a}\right| \ee^{2\zeta}.
$$
Using the symmetry \(\zeta(t+T/2)=-\zeta(t)\) we have 
$$
 \int_{t_0}^{t_0+T/2} \ee^{2\zeta(t+T/2)} dt =  \int_{t_0}^{t_0+T/2} \ee^{-2\zeta(t)}\,dt, 
$$
hence
\begin{align*}
\int_{t_0}^{t_0+T} \ee^{2\zeta(t)}\,dt  &= \int_{t_0}^{t_0+T/2}\!\!\bigl(\ee^{2\zeta(t)}+\ee^{-2\zeta(t)}\bigr)\,dt
\\&= 2\int_{t_0}^{t_0+T/2}\!\cosh(2\zeta(t))\,dt.
\end{align*}
Choosing now $t_0$ such that $z(t_0) = z_{min}$, we have $\zeta(T_0) = -A$ (the amplitude).
Also we have $\dot \zeta = \dot z \geq 0$ on the interval $[0,T/2]$, therefore 
$$
  dt =\frac{dz}{\sqrt{1-2U_{a,b}(z)}} = \frac{d\zeta}{\sqrt{1-2 |ab| \cosh (2\zeta) }},
$$
and we have thus
$$
\int_{t_0}^{t_0+T/2}\!\cosh(2\zeta(t))\,dt
= \int_{-A}^{A}\frac{\cosh(2\zeta)}{\sqrt{1-2|ab|\cosh(2\zeta)}}\,d\zeta.
$$
Therefore
\[
\int_{t_0}^{t_0+T} \ee^{2\zeta(t)}\,dt
= 4\int_{0}^{A}\frac{\cosh(2\zeta)}{\sqrt{1-2|ab|\cosh(2\zeta)}}\,d\zeta
= M,
\]
and finally
$$
\int_{t_0}^{t_0+T} \ee^{2\,z(t)}  dt =  \left| \frac{b}{a}\right|  \int_{t_0}^{t_0+T} \ee^{2\zeta(t)}\,dt
= \left| \frac{b}{a}\right| M.
$$
A similar argument proves that 
$$
\int_{t_0}^{t_0+T} \ee^{-2\,z(t)}  dt =  \left| \frac{a}{a}\right| M.
$$
\end{proof}

\medskip

The functions \(M(k)\) and \(H(k)\) are expressed in terms of elliptic integrals in Section~\ref{sec.TMHelliptic}. Their basic properties are developed in Section~\ref{sec.TMHproperties}.

\subsection{A geodesic rendez-vous }
\label{sec.rendezvous}

In this section, we prove that a generic geodesic segment of modulus $k$ cannot be minimizing if its length exceeds $T(k)$. Given such a segment, either its endpoints are conjugate points, or one can construct another geodesic segment of the same length joining the same endpoints.We formulate this in the next theorem, which is due to Coiculescu--Schwartz, \cite[Theorem 2.3 and Corollary 2.5]{CoiculescuSchwartz}. A special case was proved by Grayson in  \cite{Grayson}. 

\begin{theorem}\label{thm:rendezvous}
Let $\gamma(t)=(x(t),y(t),z(t))$ be a generic unit-speed geodesic in $\SOL$ with modulus $k\in (0,1)$ and period $T=T(k)$, and fix $t_1\in\mathbb R$. Then:
\begin{enumerate}[(i)]
\item If $\dot z(t_1)=0$, i.e. $t_1$ is a critical point of $\gamma$, then $\gamma(t_1)$ and $\gamma(t_1+T)$ are conjugate along $\gamma$.
\item   If $\dot z(t_1)\neq 0$, there exist two distinct geodesic arcs from $\gamma(t_1)$ to $\gamma(t_1+T)$, both having the same modulus $k$ and the same length, equal to $T$.
\end{enumerate}
\end{theorem}

\begin{proof}
We divide the proof into five steps. Let $\gamma$ be an arbitrary geodesic in $\SOL$ with principal constants of motion $a,b$, and fix $t_1\in\R$.

\medskip
\textit{Step 1}   Define a new curve $\gamma^{\star}(t):= (x^*(t),y^*(t),z^*(t)) $ by
\begin{equation}\label{def.partner}
\begin{cases}
  x^*(t)=2x(t_1)-x(2t_1-t),\\
  y^*(t)=2y(t_1)-y(2t_1-t),\\
  z^*(t)=z(2t_1-t).
\end{cases}
\end{equation}
We prove that $\gamma^{\star}$ satisfies the geodesic equations (Lemma \ref{zpp}): 
Let $u(t):=2t_1-t$, then  $\dot u(t)=-1$ and 
\[
 \dot x^*(t)=-\frac{d}{dt}\left( x(u(t))\right) =+\dot x(u(t))=a\,\ee^{2z(u(t))}=a\,\ee^{2z^*(t)},
\]
and similarly $\ \dot y^*(t)=b\,\ee^{-2z^*(t)}$. Moreover  
$$\dot z^*(t)=\dot z(u(t))\cdot \dot u(t)=-\dot z(u(t)), $$
therefore 
\[
\ddot z^*(t)=-\ddot z(u(t))\cdot \dot u(t) =\ddot z(u(t))=-a^{2}\ee^{2z^*(t)}+b^{2}\ee^{-2z^*(t)}.
\]
Hence   $\gamma^{\star}$ is also geodesic; with the same principal constants of motion  $a$ and $b$.

\medskip

\textit{Step 2.}  At $t=t_1$,  we have 
\begin{equation} \label{def.partner2}
  \gamma^{\star}(t_1)=\gamma(t_1) \quad  \text{and}  \quad 
 \dot\gamma^{\star}(t_1) =   (\dot\gamma(t_1)^{\dagger} = (\dot x(t_1),\dot y(t_1),-\dot z(t_1))
\end{equation}
(recall that ${\dagger}$ flips the $z$–component of a tangent vector). 
 In particular  $\|\dot\gamma^{\star}(t_1)\|= \|\dot\gamma(t_1)\|$, and since geodesics have constant speed, it follows that $\|\dot\gamma^{\star}(t)\|=\|\dot\gamma(t)\|$ for all $t\in \R$.

\medskip 

{Henceforth} we assume that $\gamma$ is a unit-speed \emph{generic} geodesic; then $\gamma^{\star}$ is also unit speed and generic. Moreover $\gamma$ and $\gamma^{\star}$ have the same average height $h$ and the same modulus $k$.

\medskip 

\textit{Step 3.} Applying  Theorem~\ref{th:spiraling}, we have 
$$
  \gamma^{\star}(t_1+T) = \gamma^{\star}(t_1) + \sgn(ab)\,M(k)\,(b,a,0)
  = \gamma(t_1) + \sgn(ab)\,M(k)\,(b,a,0) = \gamma(t_1+T),
$$
 If $\dot z(t_1)\neq 0$, then $\dot\gamma^{\star}(t_1)\neq\dot\gamma(t_1)$ by \eqref{def.partner2};
hence the arcs $\gamma|_{[t_1,t_1+T]}$ and $\gamma^{\star}|_{[t_1,t_1+T]}$ are distinct geodesic arcs.  
The proof of (ii) is complete.

\medskip

\textit{Step 4.} Assume from now on   $\dot z(t_1) = 0$, then $\gamma$ and $\gamma^{\star}$  satisfy the same initial data by \eqref{def.partner2}, therefore  $\gamma(t) = \gamma^{\star}(t)$ for all $t$. 

To deal with this situation, we slightly shift the time $t_1$ by a small number $s$ and let $s\to 0$.
More explicitly:  let us define the one parameter family of geodesics 
\[
 \psi(t,s) = \bigl(2x(t_1+s)-x(2(t_1+s)-t),\ 2y(t_1+s)-y(2(t_1+s)-t),\ z(2(t_1+s)-t)\bigr).
\]
Observe that 
$
 \psi(t,0) =  \gamma^{\star}(t) = \gamma(t) 
$
for all $t$, so 
\[
 J(t)=\left.\frac{\partial}{\partial s}\right|_{s=0}\psi(t,s)
\]
is a Jacobi field along $\gamma$. A direct computation gives, for all $t$,
\[
  J(t)=\bigl(2\dot x(t_1)-2\dot x(2t_1-t),\ \ 2\dot y(t_1)-2\dot y(2t_1-t),\ \ 2\dot z(2t_1-t)\bigr).
\]
In particular $J(t)$ is a non trivial Jacobi field since its last component $2\dot z(2t_1-t)$ is non constant. 

\medskip

\textit{Step 5.}   We claim that the Jacobi Field $J$ vanishes at $t=t_1$ and $t = t_1+T$. 
Indeed, we have at $t = t_1$
$$
  J(t_1)=(0,0,2\dot z(t_1))=0,
$$
since we are assuming $\dot z(t_1) = 0$.  Similarly,  $J(t+T) = 0$ by periodicity. 
We have constructed a non trivial jacobi field along $\gamma$ that vanishes at $\gamma(t_1)$ and $\gamma(t_1+T)$, proving (i). 
\end{proof}

\medskip 

\begin{definition} \rm 
Following \cite{CoiculescuSchwartz}, the geodesic $\gamma^{\star}$ constructed in the  above proof (i.e. by \eqref{def.partner})
is called the \emph{partner of $\gamma$ at time $t_1$}.\index{geodesic!partner} 
It can be defined either by the explicit formula \eqref{def.partner} or equivalently by the initial condition \eqref{def.partner2}.
\end{definition}

\begin{center}
\begin{figure}[htbp]
\captionsetup[figure]{labelformat=empty}
  \centering
  \includegraphics[width=0.4\linewidth, trim=1.2cm 4.5cm 1.5cm 4.61cm, clip]{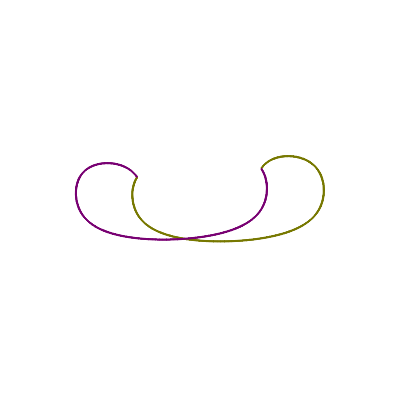}
  \caption*{A generic geodesic segment and its partner over one period.}
  \label{fig:partners}
\end{figure}
\end{center}

\subsubsection{A geometric viewpoint on the rendez-vous}

The partner geodesic $\gamma^\star$ is defined from $\gamma$ either by the explicit formulas \eqref{def.partner} or, equivalently, by solving the geodesic equations after replacing the initial velocity by its vertical reflection. 
It might therefore seem quite accidental that $\gamma$ and $\gamma^\star$ meet a second time—especially at the prescribed time $T(k)$.  The following geometric argument, based on the geometry of Grayson cylinders, explains this recurrence.

\medskip

Grayson cylinders are determined by three parameters $a$, $b$, and $c$. 
However, the parameter $c$ plays a geometric role that is quite different from that of the main parameters $a$ and $b$.
For definiteness assume $a>0$ and $b>0$. By \eqref{eq:hkab-summary}, the pair $a,b$ determines the modulus $k\in(0,1)$ and the average height $h$, and conversely.

In particular, $k$ and $h$ do not depend on $c$ and  all Grayson cylinders with the same principal parameters $a,b$ share the same equatorial plane $\{z=h\}$ and the same amplitude $A(k)$. Now we have the following 

\newpage 

\begin{lemma}
Fix $a>0$ and $b>0$, and let $p=(x,y,z)\in\SOL$. Then:
\begin{enumerate}[(i)]
\item If $|z-h|>A(k)$, there is no Grayson cylinder with principal parameters $a,b$ passing through $p$.
\item If $|z-h|<A(k)$, there are exactly two distinct Grayson cylinders with principal parameters $a,b$ passing through $p$.
\item If $|z-h|=A(k)$, there is exactly one such cylinder through $p$, and $p$ is a critical point  on that cylinder.
\end{enumerate}
\end{lemma}
In case (iii) one may say that the two Grayson cylinders coalesce into a single one at $p$.

\begin{proof}
Statement (i) is immediate from Lemma \ref{boundsGrayson}. 
To prove (ii) observe that  $|z-h| < A(k)$  if and only if $2U_{a,b}(z)< 1$, so
\[
  c = a x - b y \pm \sqrt{\,1-2U_{a,b}(z)}.
\]
This gives two distinct values of $c$, hence two distinct cylinders. The third case is similar; 
$|z-h|=A(k)$ if and only if $2U_{a,b}(z)=1$, yielding a unique value  $c ax-by$.
\end{proof}

\medskip

Chose now a point $p = (x,y,z)$ with $|z-h|\leq A(k)$ and denote by  $\mathcal G$ and $\mathcal G^{\star}$ 
the two  Grayson cylinders with parameters $a,b$ containing $p$.  To be  precise, $\mathcal G = \mathcal G_{a,b,c}$ and $\mathcal G^{\star} = \mathcal G_{a,b,c^{\star}}$ with 
$$
   c = a x - b y + \sqrt{\,1-2U_{a,b}(z)},  \quad   c^\ast = a x - b y - \sqrt{\,1-2U_{a,b}(z)}.
$$
We  claim that at the point $p=(x,y,z)\in\mathcal G_{a,b,c}\cap\mathcal G_{a,b,c^{\star}}$, 
we have
$$
  \eta^{\star}_p = \eta_p  \qquad \text{and} \qquad    \xi^{\star}_p = \xi_p^{\dagger}. 
$$
This is indeed obvious for the field $\eta$ and for the field $\xi$ , using 
$$
  (c^{\star}-ax+by) = - (c-ax+by). 
$$
We have 
\begin{align*}
  \xi^{\star}_p &= a \ee^{2z}\,\frac{\partial}{\partial x} + b \ee^{-2z}\,\frac{\partial}{\partial y}  + (c^{\star}-ax+by)\,\frac{\partial}{\partial z}  
\\ &=
a \ee^{2z}\,\frac{\partial}{\partial x} + b \ee^{-2z}\,\frac{\partial}{\partial y}  - (c-ax+by)\,\frac{\partial}{\partial z}  
\\ &= \xi_p^{\dagger}. 
\end{align*}
Let us now denote by $\gamma$ the flow line of $\xi$ starting at $p$ and by $\gamma^{\star}$ the flow line of $\xi^{\star}$  at $p$. 
They are the unique geodesic starting at $p$ on their respective Grayson cylinder. 
Flowing both fields (or equivalently following both geodesics) over one period $T(k)$ will land us again on the same intersection line, with the same drift.  Specifically, it follows from Theorem \ref{th:spiraling} that 
\[
  \gamma(t_{0}+T)-\gamma(t_{0})   =  \gamma^{\star}(t_{0}+T)-\gamma^{\star}(t_{0})   =    ab M(k)\,(b,\,a,\,0).
\]
In particular we have 
$$
 \gamma(t_{0}) =  \gamma^{\star}(t_{0}) \in  \mathcal G \cap \mathcal G^{\star}  \quad \Rightarrow 
 \quad  \gamma(t_{0 }+ T) =  \gamma^{\star}(t_{0}+T).
$$

\medskip

If $p$ is a critical point, then $c^{\star}=c$, hence $\mathcal G^{\star}=\mathcal G$ and $\xi^{\star}=\xi$; the two geodesics coincide. Approaching this situation from generic points yields, in the limit, a nontrivial Jacobi field along the common geodesic.

\newpage 

\begin{center}
 \begin{figure}[htbp]
\captionsetup[figure]{labelformat=empty}
  \centering
  \includegraphics[width=0.76\linewidth]{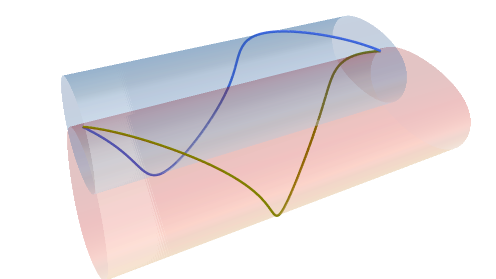}
  \caption*{Two partner segments and their respective Grayson cylinders.}
\label{fig:partnerandcylinder}
\end{figure}
\end{center}

\section{On the cut locus in $\SOL$}\label{sec:cutlocus} 

Recall that if $\gamma$ is a unit-speed geodesic in a complete Riemannian manifold
and $t\in\R$, the \emph{cut length of $t$ along $\gamma$} is defined as 
\[
  \cut_{\gamma}(t) =   \sup\bigl\{\,s\ge 0 \ \big|\ d\bigl(\gamma(t),\gamma(t+s)\bigr)=s\,\bigr\}
\in (0, \infty]. 
\]
If $\cut_{\gamma}(t) < \infty$, then the point $q=\gamma(t+s)$ is said to be a \textit{cut point} with respect to 
the base point $p = \gamma(s)$. Finally, the \textit{cut locus}\index{cut locus} of a point $p$ is the set of all the cut points in $\SOL$ with respect to the base point $p$. We denote it by
$$
  \mathrm{Cut}(p) = \{\gamma(s) \mid  \gamma \ \text{is a geodesic such that $\gamma(0) = p$ and $s= \cut_{\gamma}(0)$ }\}.
$$
The following result is classical in Riemannian geometry (see, e.g., \cite[Lemma~5.2]{CheegerEbin}):

\smallskip 

\begin{lemma}\label{lem.cutlocus}
If  $s:=\cut_{\gamma}(t)<\infty$, then $s$ is the smallest positive number such that 
\begin{enumerate}[\quad (i)]
\item  either $\gamma(t)$ and $\gamma(t+s)$ are conjugate along $\gamma|_{[t,\,t+s]}$;
\item or there exists a geodesic arc of length $s$, distinct from $\gamma|_{[t,\,t+s]}$, joining $\gamma(t)$ to $\gamma(t+s)$.
\end{enumerate}
\end{lemma}

\medskip

Combining this classical result with Theorem \ref{thm:rendezvous} immediately yields:

\begin{corollary}\label{cor:nonmin-after-T}
Let $\gamma$ be a generic unit-speed geodesic with modulus $k\in(0,1)$ and period $T=T(k)$.
If \  $t_2-t_1>T$, then
\[
  d\bigl(\gamma(t_1),\gamma(t_2)\bigr)\;<\;t_2-t_1,
\]
i.e.\ the segment $\gamma|_{[t_1,t_2]}$ is not minimizing.
\end{corollary}
\qed 

\medskip

The converse of this corollary also holds. This is the main theorem of  Coiculescu and Schwartz in ~\cite{CoiculescuSchwartz}, although that  paper  contains several other deep results. The proof is rather elaborate,  and we recall it here without proof.

\begin{theorem}\label{th.MainCS}
A generic geodesic segment is minimizing between its endpoints if and only if 
its length does not exceed $T(k)$, where $k$ is the modulus of the geodesic.
\end{theorem}

\medskip 

 \begin{corollary}\label{cor:cutlength}
{\rm (A.)} \ For any unit-speed geodesic $\gamma$ in $\SOL$ and any $t\in\mathbb{R}$,
$$
\cut_{\gamma}(t) = 
\begin{cases} \infty, & \text{if $\gamma$ is vertical or hyperbolic},\\[2pt]
   T(k), & \text{if $\gamma$ is generic with modulus $k$},\\[2pt]
   \sqrt{2}\,\pi, & \text{if $\gamma$ is horizontal}.
\end{cases}
$$
{\rm (B.)} \  Moreover, the cut locus of any point $p$ is contained in the horizontal plane through $p$. In particular, $\mathrm{Cut}(0)\subset \Pi_0$.
\end{corollary}

\smallskip 

Note that this result shows in particular that the \emph{global injectivity radius} of $\SOL$ is 
$\sqrt{2}\,\pi$, that is, every geodesic segment of length $\leq \sqrt{2}\,\pi$ is minimizing, 
whereas beyond this threshold some geodesics cease to be minimizing.

\smallskip

\begin{proof}
{\rm (A.)} \  A vertical and hyperbolic geodesic lie in totally geodesic  planes $\mathcal{H}'_{y_0}$ or $\mathcal{H}''_{x_0}$, which are  isometric to the hyperbolic plane. In ${H}^2$, there exists exactly one geodesic through any pair of distinct points, hence the cut length is $\infty$.

\smallskip 

For generic geodesics, Corollary~\ref{cor:nonmin-after-T} gives $\cut_{\gamma}(t)\le T(k)$. The converse inequality (minimality up to $T(k)$) is exactly Theorem~\ref{th.MainCS}. Therefore $\cut_{\gamma}(t)=T(k)$.

\smallskip 
Horizontal geodesics are geometrically equivalent to the unit–speed line
$\alpha(t)=\frac{1}{\sqrt{2}}(t,t,0)$.
Along $\alpha$, the vector field
$$
  J=\bigl(1-\cos(\sqrt{2}\,t)\bigr)(X-Y)+\sin(\sqrt{2}\,t)\,Z
$$
is a Jacobi field \,  (\textit{short proof}:  using the Levi–Civita connection, one checks that $\nabla_{\alpha'}(X-Y)=Z$ and $\nabla_{\alpha'}Z = Y -X$, so the Jacobi equation for  $J(t)=p(t)(X-Y)+q(t)Z$ along the geodesic $\alpha$, which is 
$\nabla_{\alpha'}\nabla_{\alpha'}J + R(J,\alpha')\alpha' = 0$, reduces to the system
$p''-2q'=0$ and $q''+2p'-2q=0$ which is  satisfied for $p(t)=1-\cos(\sqrt{2}\,t)$ and $q(t)=\sin(\sqrt{2}\,t)$).

\smallskip

Furthermore, $J$ vanishes at $t=0$ and $t=\sqrt{2}\pi$;
hence the point $P=(\sqrt{2}\pi,\sqrt{2}\pi,0)$ is conjugate to $0$ along $\alpha$,
and therefore $\cut_{\gamma}(t)\le \sqrt{2}\pi$.

To prove equality, suppose there exists a shorter geodesic, say $\beta$, joining the origin $0$ to the point $P = (\sqrt2\pi,\sqrt2\pi,0)$.
Then $\beta (0) = 0$ and $\beta (s) =P$ for some $s < \sqrt2\pi$. Then $\beta$ must be a generic geodesic (since clearly $\alpha$ is the only vertical geodesic from $0$ to $P$), thus by   Corollary \ref{cor:nonmin-after-T} we must have $s = T(k)$, where $k\in (0,1)$ is the modulus of $\beta$.  Using the  lower bound $T(k)>\sqrt{2}\,\pi$ proved in \S\ref{sec.TMHproperties}, we get 
$$
  s = T(k) > \sqrt{2}\,\pi > s,
$$
which is impossible. Hence $\cut_{\alpha}(t)=\sqrt{2}\,\pi$.

\medskip 

{\rm (B.)} \ The last statement is trivial for horizontal geodesics. For vertical and hyperbolic geodesics there is nothing to prove, since there is no cut point in that case. For generic geodesics, Theorem~\ref{th:spiraling} shows that $\gamma(t)$ and $\gamma(t+T(k))$ have the same height,
\end{proof}

\medskip

We refer to \cite{CoiculescuSchwartz} for a more detailed description of the geometry  of the cut locus.

\begin{remark} \rm 
In \cite{Coiculescu21}, 
M.~Coiculescu considered a generalization of Theorem~\ref{th.MainCS} to the so-called \emph{Bianchi groups of Type~VI}, endowed with a left-invariant Riemannian metric which is isometric to $\R^3$ with the metric tensor 
$$
  ds^2 = \ee^{-2z}\,dx^{2} + \ee^{2\alpha z}\,dy^{2} + dz^{2},
$$
(the case $\alpha =1$ corresponds to $\SOL$).
He proved that Theorem~\ref{th.MainCS} still holds for $\alpha=1/2$;   the main difficulty for general $\alpha$ lies in the complexity of the  expression for the period $T(k)$.
\end{remark}

\subsection{Asymptotic distances in the ground plane} 

We conclude with an asymptotic estimate for large horizontal distances.

\begin{theorem}\label{th.largedistance}
Fix $\theta\in(0,\pi/2)$ and set $P_\lambda=\lambda(\cos\theta,\sin\theta,0)$ with $\lambda>0$. Then
\[
  d(0,P_\lambda)=4\log\lambda+O(1) 
\quad \text{as }\lambda\to\infty,
\]
where the error term is uniformly bounded in $\lambda$, with a bound depending only on $\theta$.
\end{theorem}

\smallskip

\begin{proof}
\emph{Step 1.} Since $H:(0,1)\to(\pi,\infty)$ is strictly increasing and onto, for $\lambda$ sufficiently large there exists a unique $k\in(0,1)$ such that
\begin{equation}\label{eq.target-H}
  H(k)=\lambda\sqrt{\sin\theta\,\cos\theta}.
\end{equation}

\emph{Step 2.} Define $a>0$ and $b>0$ by 
\[
  a^{2} = \tan(\theta) \frac{1-k^{2}}{2(1+k^{2})},\qquad
  b^{2} = \cot(\theta) \frac{1-k^{2}}{2(1+k^{2})}.
\]
Then 
\[
 \frac{a}{b}=\tan\theta, \qquad 2ab=\frac{1-k^{2}}{1+k^{2}},
\]
and
\[
  a^{2}+b^{2}=\frac{1-k^{2}}{2(1+k^{2})}\,\frac{1}{\sin\theta\,\cos\theta}.
\]
If  $\lambda$ is  large enough, we have $a^{2}+b^{2}\leq 1$ and we can define the unit vector 
$$
  \xi=(a,b,c), \quad \text{where} \quad c = \sqrt{1-a^2-b^2}.
$$

\emph{Step 3.}  Using \eqref{eq.target-H} and $H(k)=M(k)\sqrt{ab}$,  one checks that
\[
  M(k)\,b=\lambda\cos\theta,\qquad M(k)\,a=\lambda\sin\theta.
\]
Indeed,  
$$
\begin{aligned}
(M(k)\,b)^2 &= M(k)^2\,ab\,\cot\theta  =  H(k)^2\cot\theta,\\
            &= \lambda^2\sin\theta\cos\theta\cdot\frac{\cos\theta}{\sin\theta}
              =  (\lambda\cos\theta)^2,
\end{aligned}
$$
and likewise for $M(k)\,a$.

\smallskip 

\emph{Step 4.} Let $\gamma$ be the unit-speed geodesic with  initial conditions $\gamma(0)=0$, $\dot\gamma(0)=\xi$. Applying Theorem~\ref{th:spiraling} we find 
\[
  \gamma\bigl(T(k)\bigr)=  M(k)\,(b,\,a,\,0) =\lambda(\cos\theta,\sin\theta,0)  = P_\lambda.
\]
By the Coiculescu--Schwartz Theorem (Theorem~\ref{th.MainCS}), the segment $\gamma|_{[0,T(k)]}$ is minimizing;, therefore we have
\[
 d(0,P_\lambda)=T(k).
\]

\emph{Step5. \ }  Proposition~\ref{prop:limits-k1} gives us 
$$
 T(k) = 2\log\left(\frac{1}{1-k^2}\right)+O(1), \quad \text{and} \quad 
 H(k) = \frac{1}{\sqrt{1-k^2}} (1+o(1)).
$$
Using now \eqref{eq.target-H} , we obtain 
\begin{align*}
T(k)
&=2\log\left(\frac{H(k)^2}{16}\right)+O(1)
 =2\log\left(\frac{\lambda^2\sin\theta\,\cos\theta}{16} + o(1)\right)+O(1)\\
& = 4 \log\lambda+O(1),
\end{align*}
and  we conclude
$$
 d(0,P_\lambda)=T(k) = 4 \log\lambda+O(1).
$$
\end{proof}

\begin{remark}\rm
\begin{enumerate}[(1)]
\item The threshold $\lambda\ge \pi/\sqrt{\sin\theta\,\cos\theta}$ ensures that \eqref{eq.target-H} has a solution, but the additional condition $a^2+b^2 \leq 1$ imposes the sharper lower bound  
\[
 \lambda_\ast(\theta):=\frac{H(k_\theta)}{\sqrt{\sin\theta\,\cos\theta}},\qquad
 k_\theta=\sqrt{\frac{1-2\sin\theta\cos\theta}{1+2\sin\theta\cos\theta}}.
\]

\item At $\lambda=\lambda_\ast(\theta)$ one has $a^2+b^2=1$ and hence $c=0$, corresponding to the critical line of a Grayson cylinder. In this case the point $P_{\lambda_\ast(\theta)}$ is conjugate to $0$ along $\gamma$. 

\item For $\theta=\pi/4$ we obtain $\lambda_\ast=\sqrt{2}\pi = T(0)$. 

\item As $\theta\to 0$ or $\theta\to\pi/2$, one has $\lambda_\ast(\theta)\to\infty$. In the limiting axial directions
$\theta=0$ or  $\theta=\pi/2$, we leave the generic regime and the unique geodesic from the origin to $P_\lambda$ is a  hyperbolic line. 
\end{enumerate}
\end{remark}

\appendix

\section{Some important integrals.} 

In this technical section we discuss key properties of the period and  the horizontal drift as functions of the modulus. We begin with a brief review of elliptic integrals, which play a basic role in this context. 

\subsection{The Legendre elliptic integrals } \label{sec.EK}

The \emph{Legendre complete elliptic integrals}\index{elliptic integrals} 
of the first and second kinds are defined for $0\le k<1$ by
\begin{equation}\label{eqKELegendre}
  K(k)=\int_{0}^{\pi/2}\frac{d\theta}{\sqrt{1-k^{2}\sin^{2}\theta}},
  \qquad
  E(k)=\int_{0}^{\pi/2}\sqrt{1-k^{2}\sin^{2}\theta}\,d\theta.
\end{equation}
Alternatively, using the substitution $u = \sin \theta$, they can also be written as:
\begin{equation} \label{eqKEalgebrique}
 K(k) = \int_0^1 \frac{du}{\sqrt{(1 - u^2)(1 - k^2 u^2)}} , \qquad
 E(k) = \int_0^1 \sqrt{\frac{1 - k^2 u^2}{1 - u^2}} \, du.
\end{equation}

\medskip

\textbf{Basic properties}
\begin{enumerate}[(i)]
 \item For all $ k \in (0,1) $, the elliptic integrals satisfy
$$
  E(k) < \frac{\pi}{2} < K(k).
$$
\item $ K(k) $ is strictly increasing and $ E(k) $ is strictly decreasing.
\item They satisfy the following obvious limits:
$$
K(0) = E(0) = \frac{\pi}{2}, \qquad \lim_{k \to 1} K(k) = \infty, \qquad \lim_{k \to 1} E(k) = 1.
$$
\item More precisely, we have the finer asymptotics as $k\to1$ 
$$
  K(k) =  \frac{1}{2} |\log(1-k^{2})| + O(1),
$$
see  \cite[Eq.~(3.8.26), p.~75]{Lawden}.
\item Their derivatives with respect to $k$ are 
$$
\frac{dK}{dk} = \frac{E(k) - (1 - k^2)K(k) }{k(1 - k^2)}, \qquad
\frac{dE}{dk} = \dfrac{1}{k} \left( E(k) - K(k) \right), 
$$
see  \cite[\S 3.10]{Lawden}.
\end{enumerate}

\medskip

We will also need the following inequality: 
\begin{lemma}\label{lem.EkK}
For $0<k<1$, one has $E(k)>(1-k^{2})\,K(k)$.
\end{lemma}

\begin{proof} This follows from the above formula for the derivative $K'(k)$ since $K(k)$ is strictly increasing. 
However a direct proof is easy. From the definition \eqref{eqKELegendre}, we have 
\begin{align*}
E(k)-(1-k^{2})K(k) &=\int_{0}^{\pi/2}\!\left(\sqrt{1-k^{2}\sin(\theta)^2}
-\frac{1-k^{2}}{\sqrt{1-k^{2}\sin(\theta)^2}}\right)\!d\theta
\\ &=\int_{0}^{\pi/2}\frac{(1-k^{2}\sin(\theta)^2) -(1-k^{2})}{\sqrt{1-k^{2}\sin(\theta)^2}}\,d\theta.
\\&= \int_{0}^{\pi/2}\frac{k^{2}\cos(\theta)^{2}}{\sqrt{1-k^{2}\sin^{2}\theta}}\,d\theta>0.
\end{align*}
\end{proof}

And  the following identity: 
\begin{lemma}\label{lem.K-aux}
For any $0\le k<1$, we have 
\[
\int_{0}^{1}\frac{1+k^{2}u^{2}}{\sqrt{1-u^{2}}\,(1-k^{2}u^{2})^{3/2}}\,du
 = \frac{2E(k)}{1-k^{2}}-K(k).
\]
\end{lemma}

\begin{proof}
Using Lebesgue's dominated convergence theorem, we may compute the derivative of $K$ 
by differentiating under the integral sign:
$$
  K'(k) =  \int_0^1\frac{\partial}{\partial k} \left(   \frac{1}{\sqrt{(1 - u^2)(1 - k^2 u^2)}}\right) \,du
=\int_{0}^{1}\frac{k\,u^{2}}{\sqrt{1-u^{2}}\,(1-k^{2}u^{2})^{3/2}}\,du.
$$
Using now the identity $1+k^{2}u^{2}=(1-k^{2}u^{2})+2k^{2}u^{2}$, we have 
\begin{align*}
\int_{0}^{1}\frac{1+k^{2}u^{2}}{\sqrt{1-u^{2}}\,(1-k^{2}u^{2})^{3/2}}\,du
 &= \int_{0}^{1}\frac{(1-k^{2}u^{2})+2k^{2}u^{2}}{\sqrt{1-u^{2}}\,(1-k^{2}u^{2})^{3/2}}\,du
\\[2mm]
&=\int_{0}^{1}\frac{du}{\sqrt{(1-u^{2})(1-k^{2}u^{2})}}
+2k^{2}\!\int_{0}^{1}\frac{u^{2}}{\sqrt{1-u^{2}}\,(1-k^{2}u^{2})^{3/2}}\,du \\[2mm]
&= K(k)+2k\,K'(k)  \\[2mm]
&=   K(k)+2k\Bigl(\frac{E(k)}{k(1-k^{2})}-\frac{K(k)}{k}\Bigr)  \\[2mm]
&=\frac{2E(k)}{1-k^{2}}-K(k),
\end{align*}
\end{proof}

\subsection{Representation of $T(k)$, $M(k)$, and $H(k)$ via elliptic integrals}
\label{sec.TMHelliptic}

In sections \ref{sec.discussZ} and \ref{sec.winding} we obtained the following  integral expressions for the period and the horizontal drift of a  generic geodesic \(\gamma\) with principal constants of motion $a,b$:
$$
  T(k) = 4\int_{0}^{A(k)} \frac{dz}{\sqrt{\,1-2|ab|\,\cosh(2z)\,}},
$$
and
$$
  H(k) := \sqrt{|ab|}\,M(k), \quad \mathrm{with} \quad M(k) = 4\int_{0}^{A(k)} \frac{\cosh(2z)\,dz}{\sqrt{\,1-2|ab|\,\cosh(2z)\,}},
$$
where $A(k) = \arctanh(k)$ is the amplitude.
These integrals are tricky to deal with directly because the integration interval depends on $k$ and shrinks to a point as $k\to 0$. In particular, the limit $\lim\limits_{k\to 0^{+}}T(k)$ cannot be  read off from the definition without first rescaling the variable. To solve this problem, we first rewrite these integrals in terms of elliptic integrals:

\medskip

\begin{proposition}\label{prop.TMH-elliptic}
With $K,E$ the complete elliptic integrals of the first and second kind, we have 
\begin{equation}\label{eq:T-and-M}
  T(k)=\sqrt{8(1+k^{2})}\,K(k),  \qquad
  M(k)=\frac{\sqrt{8(1+k^{2})}}{1-k^{2}}\Bigl(2E(k)-(1-k^{2})K(k)\Bigr),
\end{equation}
and
\begin{equation}\label{eq:H-formula}
   H(k)=\frac{4E(k)-2(1-k^{2})K(k)}{\sqrt{1-k^{2}}}.
\end{equation}
\end{proposition}

\begin{proof}  
For the computation of these integrals, we  use  the substitution:
$$
z = \frac{1}{2} \log\left( \frac{1 + ku}{1 - ku} \right), \qquad  dz = \dfrac{k\, du }{1 - k^2 u^2},
$$
so that $u$ ranges from $0$ to $1$ as $z$ ranges from $0$ to $A(k)$. We then  have the relations: 
\begin{equation}\label{relzu}
 \ee^{2z} = \frac{1 + ku}{1 - ku}, \qquad  \ee^{-2z} = \frac{1 - ku}{1 + ku}, 
\end{equation}
and
\begin{equation}\label{rel2zu}
\cosh(2z) = \frac{1 + k^2 u^2}{1 - k^2 u^2}, \qquad 
\left( 1-2|ab|\,\cosh(2z)\right)  =\frac{2k^{2}}{1+k^{2}}\cdot\frac{1-u^{2}}{1-k^{2}u^{2}}.
\end{equation}
We therefore have 
$$
  T(k) = 4 \int_0^{A(k)} \frac{dz}{\sqrt{1 - 2 |ab| \cosh(2z)}} 
  = 4k \int_0^1 \frac{k du}{(1 - k^2 u^2) \, \sqrt{\dfrac{2k^{2}}{1+k^{2}}\cdot\dfrac{1-u^{2}}{1-k^{2}u^{2}}} } ,
$$
and with  \eqref{eqKEalgebrique} we recognize 
$$
  T=\sqrt{8(1+k^{2})}\int_0^1\frac{du}{\sqrt{(1-u^{2})(1-k^{2}u^{2})}}
  =\sqrt{8(1+k^{2})}\,K.
$$
A similar computation gives us 
$$
 M(k) = 4\int_{0}^{A(k)} \frac{\cosh(2z)\,dz}{\sqrt{\,1-2|ab|\,\cosh(2z)\,}} 
 =\sqrt{8(1+k^{2})}\int_0^1\frac{1+k^{2}u^{2}}
{\sqrt{1-u^{2}}\,(1-k^{2}u^{2})^{3/2}}\,du,
$$
and applying Lemma \ref{lem.K-aux}, we  obtain 
$$
\int_0^1\frac{1+k^{2}u^{2}}{\sqrt{1-u^{2}}\,(1-k^{2}u^{2})^{3/2}}\,du
=\frac{2E}{1-k^{2}}-K.
$$
The result for $H(k)$ follows directly from $H=\sqrt{|ab|}\,M$ and $|ab|=\frac{1-k^{2}}{2(1+k^{2})}$.
\end{proof}

\medskip 

Computations give us now:
\begin{corollary}\label{cor.DerTMK}
The derivatives of the functions $T(k)$, $H(k)$, and $M(k)$ are:
\begin{align*}
T'(k) &= \sqrt{\frac{8}{1+k^{2}}}\,\frac{(1+k^{2})E(k)-(1-k^{2})K(k)}{k(1-k^{2})},\\[0.8em]
H'(k) &= \frac{2\bigl((1+k^{2})E(k)-(1-k^{2})K(k)\bigr)}{k(1-k^{2})^{3/2}},\\[0.8em]
M'(k) &= \sqrt{8(1+k^{2})}\left[
\frac{E(k)-(1-k^{2})K(k)}{k(1-k^{2})}
+\frac{k(3+k^{2})\bigl(2E(k)-(1-k^{2})K(k)\bigr)}{(1+k^{2})(1-k^{2})^{2}}
\right].
\end{align*}
The last derivative can be conveniently written as
\[
M'(k)=\frac{M(k)-T(k)}{2k}+\frac{k(3+k^{2})}{(1-k^{2})(1+k^{2})}\,M(k).
\]
\end{corollary}

\subsection{Basic properties of the functions  $T(k),M(k),$ and $H(k)$.}
\label{sec.TMHproperties}

We now derive some  fundamental properties of the functions $T,M,H$ from their representation in terms of elliptic functions.

\begin{proposition}
We have the following limiting values for $k\to 0$:
\begin{eqnarray*}
 \lim_{k\to 0} T(k) & \! = \! &  \lim_{k\to 0} M(k) = \sqrt{2} \, \pi, \qquad  \lim_{k\to 0} H(k) = \pi. 
\end{eqnarray*}
\end{proposition}

\begin{proof}
The formulas for $T(k), M(k)$ and $H(k)$ in Proposition  \ref{prop.TMH-elliptic} extend continously at $k=0$.
Since $K(0)=E(0)=\tfrac{\pi}{2}$, we have 
\[
\lim_{k\to 0}T(k) =\sqrt{8}\,K(0)=\pi\sqrt{2},\qquad
\lim_{k\to 0}M(k) =\sqrt{8}\,[\,2E(0)-K(0)\,]=\pi\sqrt{2}
\]
and 
$$
\lim_{k\to 0} H(k)  =\frac{4E(0)-2K(0)}{\sqrt{1-0^2}}=\pi.
$$
\end{proof}

The behavior as $k\to 1$ is described in the next result:
\begin{proposition}\label{prop:limits-k1}
$$
 \lim_{k\to 1}  \sqrt{1-k^2}\,  H(k) =   \lim_{k\to 1}  \frac{1 }{2}(1-k^2) \, M(k)  
 = 
  \lim_{k\to 1} \frac{2T(k)}{|\log (1-k^2)|} = 4.
$$
In particular
$$
  \lim_{k\to 1} T(k)  =  \lim_{k\to 1} M(k) =   \lim_{k\to 1} H(k) =  \infty.
$$
\end{proposition}

\medskip 

\begin{proof}
We use 
$$
  K(k) =  \frac{1}{2} |\log(1-k^{2})| + O(1)  \quad \text{as } k\to 1.
$$
This implies  $\lim\limits_{k\to 1} (1-k^2) K(k) = 0$ and since  $E(1) = 1$ we have
$$
  \lim_{k\to 1} \sqrt{1-k^{2}}\,H(k) =   \lim_{k\to 1} \ 4E(k)-2(1-k^{2})K(k)  = 4.
$$
We also have 
$$
   \lim_{k\to 1} \left( {1-k^{2}}\right) \,M(k)  =   \lim_{k\to 1} \  {\sqrt{8(1+k^{2})}} \, \bigl(2E(k)-(1-k^{2})K(k)\bigr)= 8,
$$
and
$$
 \frac{2\,T(k)}{|\log(1-k^{2})|} =\frac{2\sqrt{8(1+k^{2})}\,K(k)}{|\log(1-k^{2})|} \rightarrow\ 4.
$$
\end{proof}

Using the previous two results we can write
$$
  T(0) = M(0) = \sqrt{2} \pi,  \quad H(0) = \pi,  \quad 
  T(1) =  M(1) =  H(1) =  \infty.
$$

\begin{proposition}(Monotonicity)
For $k\in(0,1)$ the functions $T(k),M(k),H(K)$ are real-analytic and strictly increasing. 
\end{proposition}

\begin{proof}
Real analiticity is clear.  The derivatives have been given in Corollary \ref{cor.DerTMK}, we have 

\begin{enumerate}[1)]
\item The function  $k\mapsto T(k)=\sqrt{8(1+k^{2})}\,K(k)$ \.   is clearly strictly increasing, since it is the product of two strictly increasing functions.
\item \[
H'(k)=\frac{2}{k(1-k^{2})^{3/2}}\Bigl((1+k^{2})E-(1-k^{2})K\Bigr)>0,
\]
since  $E(k)>(1-k^{2})K(k)$ by Lemma \ref{lem.EkK}.
\item 
 And because $M\ge T$, we have 
\[
 M'(k)=\frac{M(k)-T(k)}{2k}+\frac{k(3+k^{2})}{(1-k^{2})(1+k^{2})}\,M(k)>0.
\]
\end{enumerate}
\end{proof}

Combining the previous results, we obtain: 
\begin{corollary}
The functions  $k\mapsto T(k)$  and  $k\mapsto M(k)$  define real analytic diffeomorphism from  $[0,1)$ onto $[\sqrt{2} \pi ,\infty)$, and   $k\mapsto H(k)$  define real analytic diffeomorphism from  $[0,1)$ onto $[\pi ,\infty)$.
\end{corollary}

Finally we have the following inequalities:
\begin{proposition}
For $0<k<1$ one has
\[
  M(k)\;>\;\sqrt{2}\,H(k)\;>\;T(k)\;>\;4\,A(k).
\]
Moreover, $M(0)=\sqrt{2}\,H(0)=T(0)=\sqrt{2}\,\pi$ and $A(0)=0$.
\end{proposition}

\begin{proof}
Since $H(k)=\sqrt{|ab|}\,M(k)$ and $2|ab|\in(0,1)$, we have $M(k)>\sqrt{2}\,H(k)$ for $0<k<1$ (with equality only at $k=0$).
Let us define  $f(k):=\sqrt{2}\,H(k)-T(k)$, then the inequality 
$$
  \frac{H'(k)}{T'(k)}=\sqrt{\frac{1+k^{2}}{2(1-k^{2})}}>\frac{1}{\sqrt{2}}
$$
implies 
$$
 f'(k)=T'(k)\!\left(\sqrt{2}\,\frac{H'}{T'}-1\right)>0
$$
for any $k\in (0,1)$. Since \(f(0)=0\), it follows that \(\sqrt{2}\,H(k)>T(k)\) for \(k>0\).
Finally, for \(z\in[0,A(k)]\), we have \(1-2|ab|\cosh(2z)\le1\), therefore
\[
T(k)=4\!\int_{0}^{A(k)}\frac{dz}{\sqrt{1-2|ab|\cosh(2z)}}>4A(k).
\]
Combining the three inequalities proves the Proposition. 
\end{proof}

{\small 
\begin{remark}\rm
Using Gauss’s relation between the complete elliptic integral and the arithmetic–geometric mean\index{arithmetic–geometric mean} (AGM),\footnote{For $x,y>0$, the arithmetic–geometric mean $\mathrm{AGM}(x,y)$ is the common limit of the iteration $(x,y)\mapsto(\sqrt{xy},\tfrac12(x+y))$.}  
which states that
\[
  K(k)=\frac{\pi}{2\,\mathrm{AGM}\!\bigl(1,\sqrt{1-k^{2}}\bigr)}
\]
(see \cite[Eq.~(3.9.39), p.~81]{Lawden}), together with the relation
$
k^{2}=\frac{1-2|ab|}{1+2|ab|},
$
and the homogeneity of the  AGM, we obtain after a short computation 
\[
  T(k)=\sqrt{8(1+k^{2})}\,K(k)     
       = \frac{\pi}{\mathrm{AGM}\!\bigl(\sqrt{|ab|},\,\tfrac12\sqrt{1+2|ab|}\bigr)}.
\]
This representation of the period is the viewpoint adopted in \cite{CoiculescuSchwartz}.
\end{remark}
}

\section{Geodesics in the Heisenberg ($\NIL$) geometry}
\label{sec.nil}

In this  final section we solve the geodesic equations for the standard left–invariant Riemannian metric on the three–dimensional Heisenberg group, 
the so called $\NIL$\index{NIL} geometry in Thurston's classification. Although these geodesics are well known (see, e.g., Exercise~2.90 bis, p.~88, and its solution on pp.~280–290 in \cite{GHL}), we include here a brief self–contained resolution by quadrature as we did in the case of $\SOL$. Specifically we exploit the special form of the metric tensor to identify two basic constants of motions, which then allow a successive integration of the geodesic equations. 

\medskip 

Recall that the Heisenberg group\index{Heisenberg group} is defined as 
\[
  \mathrm{Nil}
  = \left\{
  \begin{pmatrix}
  1 & x & z\\
  0 & 1 & y\\
  0 & 0 & 1
  \end{pmatrix}
  \,\middle|\, x,y,z\in\mathbb{R}\right\},
\]
endowed with the left–invariant metric
\[
  ds^2 = dx^2 + dy^2 + (dz - x\,dy)^2.
\]
Accordingly, geodesics are the critical curves of the action $\int \mathcal{L}\,dt$ with Lagrangian
\[
  \mathcal{L} =\frac12\bigl(\dot x^{\,2} + \dot y^{\,2} + (\dot z - x\,\dot y)^{2}\bigr).
\]

Applying the Euler-Lagrange Equations \eqref{EulerLagrance}, we see that 
$$
 c:= \frac{\partial\mathcal{L}}{\partial \dot z}   =\dot z-x\dot y
  \qquad \text{and} \qquad 
 b :=  \frac{\partial\mathcal{L}}{\partial \dot y} = (1+x^2)\dot y-x\dot z,
$$
are constants of motion. We then obtain the following equations. 
\begin{equation*}
  \dot y=b+c\,x,\qquad   \dot z=c+x\,\dot y=c+b\,x+c\,x^2.
\end{equation*}
The remaining Euler--Lagrange equation is  the equality of 
$$
  \frac{\partial \mathcal{L}}{\partial x}
  =-(\dot z-x\dot y)\,\dot y=-c\,\dot y
 \quad \text{and} \quad 
  \frac{d}{dt}\!\left(\frac{\partial \mathcal{L}}{\partial \dot x}\right)=\ddot x,
$$
hence
\[
  \ddot x=-c\,\dot y=-c^2\,x-bc.
\]
Finally, if the geodesic has unit speed then  $\dot x^2+\dot y^2+(\dot z-x\dot y)^2=1$;  therefore
\[
  \dot x^2+\dot y^2+c^2=1.
\]
We  sum up  our results so far in the 
\begin{lemma}
A  smooth curve $\gamma(t) = (x(t),y(t), z(t))$ in $\NIL$ is a unit speed geodesic if and only if there exists two  constants $b,c$ 
such that 
\[
\boxed{
\begin{aligned}
  \dot y &= b+c\,x & \qquad    &\dot z = c+b\,x+c\,x^2\\[2pt]
  \ddot x &= -c^2\,x-bc &   & \dot x^2+\dot y^2+c^2 = 1  
\end{aligned}}
\]
\end{lemma}
 This ODE system is easily integrated: for $c\neq0$ one obtains helical motions winding around vertical cylinders (with explicit  parametrization by trigonometric functions), while the degenerate case $c=0$ yields parabolas in vertical planes (or lines when $ab=0$). Let us perform these calculations.

\medskip

We first assume  $c\neq 0$. The general solution to the third equation  $\ddot x=-c^2 x-bc$ can be written as 
$$
 x(t) = A\cos\big(c t-\phi\big) - \frac{b}{c},
$$
for some amplitude $A\ge 0$ and phase $\phi$. We then obtain $y(t)$ by integrating 
$$
  \dot y = b+c x = c A\cos\big(c t-\phi\big),
$$  
obtaining
$$
  y(t) = A\sin\big(c t-\phi\big) + y_0.
$$
And from the vertical derivative
\begin{eqnarray*}
 \dot z(t)  &=& c+x(t)\dot y(t) = c + cx(t)  A\cos\big(c t-\phi\big),
 \\  & = & 
 c + \frac{cA^{2}}{2}\Bigl(1+\cos\bigl(2(ct-\phi)\bigr)\Bigr)
  - bA\cos(ct-\phi),
\end{eqnarray*}
we obtain by integration the slightly more complex expression: 
\[
  z(t)=z_0+\Big(c+\frac{cA^2}{2}\Big)\cdot t  +\frac{A^2}{4}\,\sin\!\big(2(ct-\phi)\big)
       -\frac{bA}{c}\,\sin\!\big(ct-\phi\big).
\]
We thus have established that for $c\neq 0$, the geodesic $\gamma$ is parametrized by  
\[
{\setlength{\jot}{12pt} 
\left\{
\begin{aligned}
x(t) &= x_0 + A\cos\big(ct-\phi\big),
\\
y(t) &= y_0 + A\sin\big(ct-\phi\big),\\
z(t) &= z_0+\Big(c+\tfrac{cA^2}{2}\Big)t  + \tfrac{A^2}{4}\sin\!\big(2(ct-\phi)\big) - \tfrac{bA}{c}\sin\!\big(ct-\phi\big), 
\end{aligned}
\right.
}
\]
where $x_0 = -b/c$ \  and  $y_0,\, z_0$  are arbitrary. 
We observe the following  three points:
\begin{enumerate}[(i)]
\item The coordinates $x(t)$ and $y(t)$ are periodic and $z(t)$  is the sum of a  linear and a  periodic function. 
\item For all $t$ we have
$$
  \left(x(t)+\frac{b}{c}\right)^2+\big(y(t)-y_0\big)^2=A^2,
$$
in particular the  geodesic $\gamma$  is winding around a vertical circular cylinder of radius $A$.
\item The amplitude is
$$
  A = \dfrac{\sqrt{1-c^2}}{|c|},
$$
this follows from 
$$
 1 = \dot x^2+\dot y^2+c^2 = (Ac)^2\sin\big(c t-\phi\big)^2 +  (Ac)^2\cos\big(c t-\phi\big)^2 + c^2.
$$
\end{enumerate}

\medskip

We now consider the degenerate case $c= 0$. In that case the equations in the previous Lemma  reduce to 
\[
\boxed{
\begin{aligned}
  \dot y &= b  & \qquad    &\dot z =  bx \\[2pt]
  \ddot x &= 0 &   & \dot x^2+\dot y^2 =  1
\end{aligned}}
\]
Thus $a = \dot  x$ is constant and we have  $a^2+b^2 = \dot x^2+\dot y^2 =  1$.
The  solutions to the above ODEs are 
\[
 x(t)=a t+x_0, \qquad 
  y(t)=b t+y_0,\qquad
  z(t)=\frac{ab}{2}\,t^2+b x_0\,t+z_0.
\]
If $ab \neq 0$, the  geodesic is then a parabola contained in a vertical plane 
$
     bx - ay = \mathrm{const.}
$
If $ab=0$ the parabola degenerates to a  line.

\end{document}